\documentclass[11pt]{article}

\usepackage{amscd,amsmath, amssymb, fancyhdr, epsfig, mathbbol, dutchcal, url}
\usepackage[matrix,arrow,curve]{xy}
\usepackage[backref=page]{hyperref}
\usepackage[T1]{fontenc}

\usepackage{tikz-cd}
\usetikzlibrary{patterns}

\renewcommand*{\backrefalt}[4]{%
	\ifcase #1 (Not cited.)%
	\or        (Cited on page~#2.)%
	\else      (Cited on pages~#2.)%
	\fi}

\hypersetup{
	colorlinks   = true,
	citecolor    = magenta,
	linkcolor    =blue,
	urlcolor     =magenta	
}

\numberwithin{equation}{section}

\newcommand{\version}{version 1.0,\ \ June 7, 2026}

\def\eqref#1{(\ref{#1})}
\newcommand{\goth}{\mathfrak}
\newcommand{\g}{{\mathfrak g}}

\newcommand{\arrow}{{\:\longrightarrow\:}}
\newcommand{\Z}{{\Bbb Z}}
\def\C{{\Bbb C}}
\newcommand{\R}{{\Bbb R}}
\newcommand{\Q}{{\Bbb Q}}

\def\1{\sqrt{-1}\:}
\newcommand{\restrict}[1]{{\left|_{{\phantom{|}\!\!}_{#1}}\right.}}
\newcommand{\cntrct}                % contraction with a vector field
{\hspace{2pt}\raisebox{1pt}{\text{$\lrcorner$}}\hspace{2pt}}
\newcommand{\calo}{\mathcal O}

\newcommand{\F}{\mathcal{F}}
\newcommand{\G}{\mathcal{G}}

\renewcommand{\L}{{L}}
\newcommand{\M}{\mathcal{M}}
\newcommand{\OP}{\mathcal{O}}
\DeclareMathOperator{\pr}{pr}
\newcommand{\p}{\mathbb{P}}
\DeclareMathOperator{\Div}{Div}

\renewcommand{\l}{\ell}

\renewcommand{\phi}{\varphi}
\renewcommand{\epsilon}{\varepsilon}
\renewcommand{\geq}{\geqslant}
\renewcommand{\leq}{\leqslant}

\newcommand{\Pic}{\operatorname{Pic}}

\newcommand{\Tot}{\operatorname{Tot}}

\newcommand{\Vol}{{\operatorname{Vol}}}

\newcommand{\Aut}{\operatorname{Aut}}

\newcommand{\Lie}{\operatorname{Lie}}

\newcommand{\Kod}{{\operatorname{Kod}}}

\newcommand{\SG}{\mathfrak{S}}

\newcounter{Mycounter}[section]
\newcounter{lemma}[section]
\renewcommand{\thelemma}{{Lemma \thesection.\arabic{lemma}}}
\newcommand{\lemma}{%
    \setcounter{lemma}{\value{Mycounter}}
    \refstepcounter{lemma}
    \stepcounter{Mycounter}
    {\noindent \bf \thelemma:\ }}

\newcounter{claim}[section]
\renewcommand{\theclaim}{{Claim \thesection.\arabic{claim}}}
\newcommand{\claim}{%
    \setcounter{claim}{\value{Mycounter}}
    \refstepcounter{claim}
    \stepcounter{Mycounter}
    {\noindent \bf \theclaim:\ }}

\newcounter{sublemma}[section]
\newcounter{corollary}[section]
\renewcommand{\thecorollary}{{Corollary \thesection.\arabic{corollary}}}
\newcommand{\corollary}{%
    \setcounter{corollary}{\value{Mycounter}}
    \refstepcounter{corollary}
    \stepcounter{Mycounter}
    {\noindent \bf \thecorollary:\ }}

\newcounter{theorem}[section]
\renewcommand{\thetheorem}{{Theorem \thesection.\arabic{theorem}}}
\newcommand{\theorem}{%
    \setcounter{theorem}{\value{Mycounter}}
    \refstepcounter{theorem}
    \stepcounter{Mycounter}
    {\noindent \bf \thetheorem:\ }}

\newcounter{conjecture}[section]
\newcounter{proposition}[section]
\renewcommand{\theproposition}
      {{Proposition \thesection.\arabic{proposition}}}
\newcommand{\proposition}{%
    \setcounter{proposition}{\value{Mycounter}}
    \refstepcounter{proposition}
    \stepcounter{Mycounter}
    {\noindent \bf \theproposition:\ }}

\newcounter{definition}[section]
\renewcommand{\thedefinition}
      {{Definition~\thesection.\arabic{definition}}}
\newcommand{\definition}{%
    \setcounter{definition}{\value{Mycounter}}
    \refstepcounter{definition}
    \stepcounter{Mycounter}
    {\noindent \bf \thedefinition:\ }}

\newcounter{example}[section]
\renewcommand{\theexample}{{Example \thesection.\arabic{example}}}
\newcommand{\example}{%
    \setcounter{example}{\value{Mycounter}}
    \refstepcounter{example}
    \stepcounter{Mycounter}
    {\noindent \bf \theexample:\ }}

\newcounter{remark}[section]
\renewcommand{\theremark}{{Remark \thesection.\arabic{remark}}}
\newcommand{\remark}{%
    \setcounter{remark}{\value{Mycounter}}
    \refstepcounter{remark}
    \stepcounter{Mycounter}
    {\noindent \bf \theremark:\ }}

\newcounter{problem}[section]
\newcounter{question}[section]
\renewcommand{\thequestion}{{Question \thesection.\arabic{question}}}
\newcommand{\question}{%
    \setcounter{question}{\value{Mycounter}}
    \refstepcounter{question}
    \stepcounter{Mycounter}
    {\noindent \bf \thequestion:\ }}

\newcommand{\proof}{\noindent{\bf Proof:\ }}

\newcommand{\pstep}{\noindent{\bf Proof. Step 1:\ }}

\makeatletter

\@addtoreset{equation}{section} \@addtoreset{footnote}{section}

\def\x@arrow{\DOTSB\Relbar}
\def\xlongrightarrowfill@{\arrowfill@\relbar\relbar\longrightarrow}
\newcommand{\xlongrightarrow}[2][]{%
        \ext@arrow 0099\xlongrightarrowfill@{#1}{#2}}
\makeatother

\def\blacksquare{\hbox{\vrule width 5pt height 5pt depth 0pt}}
\def\endproof{\blacksquare}

\begin{document}

%%%%%%%%%%%%%%%%%%%%%%%%%%%%%%%%%%%%%%%%%%%%%%%%%%%%%%%%%%%%
\begin{center}
{\LARGE\bf Projective subvarieties of 
Bogomolov-Guan manifolds and quasi-diagonals in products of elliptic curves
\\[4mm]
}
%%%%%%%%%%%%%%%%%%%%%%%%%%%%%%%%%%%%%%%%%%%%%%%%%%%%%%%%%%%%%

Leila Abubakarova, Alexandra Kuznetsova, Misha Verbitsky\footnote{M. V. acknowledges support 
of CNPq - Process 310952/2021-2,
and FAPERJ SEI-260003/000410/2023. \\

}

\end{center}

{\small \hspace{0.10\linewidth}
\begin{minipage}[t]{0.85\linewidth}
{\bf Abstract.} 
We study complex subvarieties
in certain non-K\"ahler holomorphically symplectic 
manifolds $X$, called the Bogomolov-Guan manifolds. 
Let $E$ be an elliptic curve, $L$ an ample line bundle on
$E$, $S\subset E^2$ a complex curve, and $p_1, p_2$ the
corresponding projections of $S$ to $E$.  The curve $S$ 
is called {\bf a quasi-diagonal} if $p_1^*L\otimes p_2^* L^{-1}$
is a torsion line bundle. We show that there
are at most countably many quasi-diagonals for any
$(E,L)$. Using the quasi-diagonals, 
we classify the projective subvarieties in the Bogomolov-Guan 
manifold. The Bogomolov-Guan manifold
is equipped with a Lagrangian fibration
$\pi:\; X \to \C P^n$. We show that
an irreducible complex subvariety $Z\subset X$
is Moishezon if and only if $\pi(Z)$ is
a point or a certain complex curve which 
is described in terms of quasi-diagonals. 
This is used to prove that for a general
Bogomolov-Guan manifold, any projective
subvariety belongs to a fiber of $\pi$.
\end{minipage}
}

\tableofcontents

%%%%%%%%%%%%%%%%%%%%%%%%%%%%%%%%%%%%%%%%%%%%%%%%%%%%%%%%%%%%%%%%%%%%%%%%%%%%

\section{Introduction}

%%%%%%%%%%%%%%%%%%%%%%%%%%%%%%%%%%%%%%%%%%%%%%%%%%%%%%%%%%%%%%%%%%%%%%%%%%%%

%%%%%%%%%%%%%%%%%%%%%%%%%%%%%%%%%%%%%%%%%%%%%%%%%%%%%%%%%%%%
\subsection{Bogomolov-Guan manifolds}
%%%%%%%%%%%%%%%%%%%%%%%%%%%%%%%%%%%%%%%%%%%%%%%%%%%%%%%%%%%%

In an MPIM preprint \cite{_Todorov:MPIM_}, 
A. Todorov claimed that any compact,
simply connected holomorphically symplectic manifold with
$h^{2,0}=1$ is K\"ahler.
However, in \cite{_Guan1_,_Guan2_,_Guan3_}, D. Guan constructed examples of
manifolds which are compact, simply connected, holomorphically symplectic
but non-K\"ahler. These examples, constructed from nilmanifolds, were later studied further and explained by F. Bogomolov (\cite{_Bogomolov:BG_}), who used the Hilbert scheme of
the Kodaira surface. We call these complex manifolds {\bf Bogomolov-Guan manifolds} or {\bf BG-manifolds}, 
and they are the central objects of study in this paper.

Although Bogomolov's construction of the Bogomolov-Guan manifolds 
is quite intricate (see Section \ref{_section_BG_construction_} for the details)
a step-by-step analysis reveals that their geometry resembles, in many respects, that of irreducible holomorphic symplectic manifolds. In particular, these manifolds have even dimensions, and the first example of such a manifold occurs in dimension~4.

The first systematic study of the geometric properties of Bogomolov-Guan manifolds was carried out by Kurnosov and Verbitsky in \cite{_KV_deformations_}.
They proved that the holomorphically symplectic deformations of 
these manifolds are unobstructed, and demonstrated 
that they admit an analogue of the Beauville-Bogomolov-Fujiki 
form. 

Furthermore, it was shown in \cite{_BKKY_} that the algebraic reduction 
of a $2n$-dimensional Bogomolov-Guan manifold $BG$ is a morphism 
$P\colon BG\to\p^n$ to a projective space of half the dimension. 
The generic fiber of $P$ is an abelian variety, and the symplectic 
form vanishes on the fibers; thus, $BG$ admits a Lagrangian fibration. 
The study of algebraic subvarieties in \cite{_BKKY_} also revealed a 
strong restriction: any subvariety of $BG$ whose image under the 
Lagrangian fibration $P$ has dimension at least 2 cannot be projective.

%Finally, the group of biholomorphic automorphisms
%of a $2n$-dimensional Bogomolov-Guan manifold  $BG$ 
%were studied in \cite{_BKKY_}, it was proved 
%that this group is an extension of a finite group by a subgroup of 
%automorphisms of a generic fiber of Lagrangian fibration. 

%%%%%%%%%%%%%%%%%%%%%%%%%%%%%%%%%%%%%%%%%%%%%%%%%%%%%%%%%%%%
\subsection{Bogomolov-Guan precursors}\label{_subsection_BG_precursor_}
%%%%%%%%%%%%%%%%%%%%%%%%%%%%%%%%%%%%%%%%%%%%%%%%%%%%%%%%%%%%

Bogomolov's construction of Bogomolov-Guan manifolds is long and complicated. Therefore, in this paper, we focus on their precursors: these are complex manifolds that are simpler to construct but closely related to the original ones.

As in the original construction of Bogomolov, we start from a
(primary) Kodaira surface $\Kod$ (\ref{_definition_Kodaira_}),
which can be obtained as follows.
Let $\Tot^\circ(L)$ be  the total space of all non-zero vectors
in an ample bundle over an elliptic curve $E$; it is a 
principal $\C^*$-fibration over $E$. Let $\Z$ act on
$\Tot^\circ(L)$ by homothety, multiplying each vector
by a fixed complex number $\lambda$, with $|\lambda|>1$.
The Kodaira surface is defined as the quotient $\Tot^\circ(L)/\Z$;
it is a holomorphic principal elliptic fibration over an elliptic curve.

The precursor is constructed as follows. We fix a Kodaira surface 
$\Kod$ and consider its algebraic reduction $\pi\colon\Kod\to E$, where both the base $E$ and the general fiber $E_1$ are elliptic curves.
Let $\sigma\colon E^{n+1}\to E$ be the morphism 
$\sigma(x_1,\dots, x_{n+1}) = x_1+\dots+x_{n+1}$. 
We denote by $D$ the fiber of the composition 
$\sigma\circ \pi^{n+1}\colon \Kod^{n+1}\to E$ over a point $e\in E$. 
The natural action of $E_1$ on $\Kod$ induces a diagonal action on $D$. 
The quotient $\Kod^n/E_1$ is called {\bf the unrestricted Bogomolov-Guan precursor},
and $X:=D/E_1$ is a $2n$-dimensional complex manifold which 
we call a {\bf restricted Bogomolov-Guan precursor} or {\bf BG-precursor}.
We observe that the precursor depends on the Kodaira surface $\Kod$
and on the integer $n$.

The Bogomolov-Guan precursor $X$ inherits rich geometric properties. 
The holomorphic symplectic form on the Kodaira surface $\Kod$ induces 
a holomorphic symplectic form $\Omega$ on $X$. Furthermore, if we denote by $A$ the 
fiber of $\sigma\colon E^{n+1}\to E$ over $e$, then 
the precursor admits a natural map $P_X\colon X\to A\subset E^{n+1}$, which coincides with the  
Lagrangian fibration on $X$ with respect to $\Omega$. It is easy to observe that $A\cong E^{n}$.

The construction of the actual Bogomolov-Guan manifold 
$BG$ is similar, but it replaces the $(n+1)$-fold product $\Kod^{n+1}$ 
with the Hilbert scheme of $n+1$ points on $\Kod$. In addition, after taking 
the quotient by the diagonal action of $E_1$, one must pass to a specific 
finite cover to resolve the singularities. For details, see Section 
\ref{_section_BG_construction_}. 

The Bogomolov-Guan manifold $BG$ and its precursor $X$ are related via the 
action of the symmetric group $\SG_{n+1}$. Specifically, $\SG_{n+1}$ acts 
naturally on $X$, and there exists a generically finite morphism 
$Q\colon BG \to X/\SG_{n+1}$. This relationship is summarized by 
the following commutative diagram:
\[
\xymatrix{ BG \ar[dd]_{P} \ar[dr]^{Q} && X \ar[dd]^{P_X} \ar[ld]_{Q_X}\\
   &X/\SG_{n+1}&\\
   \p^n && E^n \ar[ll]_{\alpha}
}
\]
Here, $Q_X$ denotes the quotient map (which is finite), while $Q$ is generically finite. 
The morphism $\alpha\colon E^n\to \p^n$ is defined in the following way:
\[
  \alpha\colon E^n\to E^{[n+1]}; \ \ \ \ \ \alpha(x_1,\dots, x_n)= \{x_1,\dots, x_n, -x_1-\dots-x_n\}.
\]
The image of $\alpha$ can be verified to be isomorphic to $\p^n$. 
In the following sections, we demonstrate that the subvarieties of $BG$ 
contained in the fiber $P^{-1}(Z)$ share the same geometric properties 
as the subvarieties of $X$ contained in 
$P_X^{-1}(\alpha^{-1}(Z))$ for $Z\subset \p^n$.

By construction, the Bogomolov-Guan manifold
is equipped with a projective map to a finite quotient
of its precursor (\ref{_remark_subvarieties_in_BG_and_Kod_mod_E1_}). 

We are in a position to explain the starting point of the
research carried in this paper. It is not hard to see
that all positive-dimensional subvarieties of the Kodaira
surface are obtained from the fibers of the elliptic
fibration. Indeed, consider the Kodaira surface 
$\Kod=\Tot^\circ(L)/\Z$ associated
with a line bundle $L$ on an elliptic curve $E$ 
(\ref{_definition_Kodaira_}), and
let $\pi:\; \Kod \to E$ be the standard fibration.
Since the pullback of $L$ to $\Tot^\circ(L)$ is trivial,
this line bundle is flat on $\Kod$, which implies that $c_1(\pi^*(L))=0$.
In other words, the pullback of the K\"ahler
form $\omega_E$ on $E$ to $\Kod$ is exact. This implies
that no curve $S\subset \Kod$ can be transversal
to fibers of $\pi$, indeed, $\int_S\pi^*\omega_E=0$,
which cannot happen if $TS$ contains a vector which
is not tangent to the fibers of $\pi$.

This argument can be easily generalized, bringing
classification-type results for many other non-K\"ahler
complex manifolds
(\cite{_PUV:moment-angle_,_OV:Oeljeklaus-Toma_,_OVV:Oeljeklaus-Toma_flat_}).

Unfortunately, 
direct attempt to apply this method to BG-manifolds yields no result beyond what was already obtained in \cite{_BKKY_}.
Consider the Lagrangian projection $\pi:\; X \to E^n$,
and let $V\subset H^2(E^n)$ be the space generated
by the K\"ahler classes of the individual elliptic curve
components in the product. Then $\pi^*:\; V \to H^2(X)$
has a rank 1 kernel. Using the standard arguments
(\ref{_Proof_for_class_C_Theorem_},
Step 1), this can be used to show that for any positive-dimensional
K\"ahler or Fujiki class C subvariety  $Z\subset X$,
the restriction map $V \to H^2(Z)$ has the same kernel
as $\pi^*:\; V \to H^2(X)$. Translated to geometry,
this implies that $\pi(Z)$ is a curve, and its
projection to components of $E^{n+1} \supset E^n$
has the same degree (\ref{_Proof_for_class_C_Theorem_}).
This result, originally proven in \cite{_BKKY_},
gives a complete characterization of K\"ahler and Fujiki
class C subvarieties in precursors, and, therefore, in
BG-manifolds.

To continue this work further, we focus on algebraic
(projective or Moishezon) subvarieties in BG-manifolds and
their precursors. In this case, we work with pullbacks
of non-torsion line bundles instead of pullbacks of the
K\"ahler forms. It is not hard to see 
that a pullback of a non-torsion
line bundle under a dominant map
$X \to Y$ cannot become torsion,
if $X$ and $Y$ are K\"ahler or
Fujiki class C
(\ref{_pullback_vanishes_when_torsion_Remark_}). 
This gives a more granular form of
Blanchard's theorem (\ref{_Blanchard_Kahler_Theorem_})
on K\"ahlerness of principal toric bundles.
As explained in \cite{_Hofer:remarks_} (see also 
Subsection \ref{_principal_toric_Subsection_}),
holomorphic principal toric bundles on $M$
are classified by $H^1(M, {\cal T})$, where ${\cal T}$
is a sheaf of holomorphic maps to an $n$-dimensional
 compact complex torus.
The coboundary map in the standard exact sequence \eqref{_exact_Hofer_Equation_}
\[ 
0\to H^1(X, \Z_X^{2n})\stackrel \phi \to H^{0,1}(X)^n\to
H^1(X, {\Bbb T})\to H^2(X, \Z_X^{2n})\to ...
\]
produces a collection of classes in $H^2(X, \Z_X^{2n})$
called {\bf the Chern classes} of the principal toric fibration.
Blanchard proved   (\ref{_Blanchard_Kahler_Theorem_})
that the total space of a principal toric bundle 
over a K\"ahler manifold is K\"ahler if and only 
if its Chern classes are torsion. We prove that 
it is projective if and only if the corresponding classes
in $H^1(X, {\Bbb T})$ are torsion
(\ref{_projective_toric_Theorem}).
Applied to BG-precursors, this implies 
the following classification of projective and Moishezon
subvarieties in a BG-precursor $\pi:\; X \to E^n$.

\hfill

\theorem \label{_thm_leading_to_qd_}
Let $\pi:\; X \to E^n$ be a BG-precursor
associated with a curve $E$ and an ample line bundle $L$
as above. Denote by $L_{ij}$, $i, j=1, ..., n+1$ 
the bundle $p^*_i(L)\otimes p^*_j(L^{-1})$, where $p_i:\; E^{n+1}\to E$
denotes the projection to the $i$-th factor, and
the base $E^n$ of the BG-precursor is considered
as a subvariety in $E^{n+1}$.
Let $Z\subset X$ be an irreducible subvariety of
positive dimension in a BG-precursor.
Then the following are equivalent.
\begin{description}
\item[(i)]
$\pi(Z)$ is a curve, and the pullback of
$L_{ij}$ to $\pi(Z)$ is a torsion line bundle;
\item[(ii)]
$Z\subset X$ is projective.
\end{description}

\proof \ref{_proj_then_quasidiag_Theorem_},
\ref{_proj_dim=1_base_Remark_}.
\endproof

\hfill

\remark This statement is true for restricted
precursors, as well as for the unrestricted ones.

\hfill

This theorem naturally leads us to the main notion
of this paper --- quasi-diagonals in a product of elliptic curves.

%%%%%%%%%%%%%%%%%%%%%%%%%%%%%%%%%%%%%%%%%%%%%%%%%%%%%%%%%%%%
\subsection{Quasi-diagonals in a product of elliptic curves}
%%%%%%%%%%%%%%%%%%%%%%%%%%%%%%%%%%%%%%%%%%%%%%%%%%%%%%%%%%%%

Fix an elliptic curve $E$ and a line bundle $L\in \Pic_d(E)$, 
where $d\geqslant 1$ and consider 
an irreducible reduced curve $S$ in the direct product $E^n$ 
and denote by $p_i\colon S\to E$ the projection to the $i$-th 
component of the product. We say that $S$ is a {\bf quasi-diagonal 
over $(E,L)$} 
if the line bundle $p_i^*L\otimes p_j^*L^{-1}$ is a torsion 
element in $\Pic_0(S)$ for all $1\leqslant i\leqslant j\leqslant n$. 
If, moreover, $S$ lies in the fiber of the map $\sigma\colon E^n\to E$ 
such that $\sigma(x_1,\dots, x_n) = x_1+\dots+ x_n$, then
we say that $S$ is a {\bf restricted quasi-diagonal}.

As we stated in \ref{_thm_leading_to_qd_} a subvariety $Y\subset X$ is projective if 
and only if its projection $p(Y)\subset E^{n+1}$ is either a point or a 
restricted quasi-diagonal over $(E,L)$  where $L$ is a line bundle associated with $X$.

Quasi-diagonals in $E^n$ are interesting geometric objects in their own right. 
We construct several examples of such curves (see Section 
\ref{_section_examples_of_qd_}) and establish some of their properties. 
Our main result concerning quasi-diagonals is that they cannot appear 
in continuous families: for a fixed line bundle $L$, the set of 
quasi-diagonals in $E^n$ over $(E,L)$ is at most countable.

However, many questions on this subject remain open. One of the most intriguing 
problems is how to construct new examples of quasi-diagonals and whether one 
can classify them in any meaningful way. This problem is highly non-trivial 
even for quasi-diagonals in $E^2$ and a complete answer remains out of reach. 
Note that in this case a quasi-diagonal $S\subset E^2$ defines a self-correspondence 
of $E$, suggesting the use of results from the theory of correspondences. 

For instance, an important notion in the study of self-correspondences of curves
in \cite{_Bellaiche_} is an {\bf irreducible complete set}: we say 
that $Z\subset E$ is such a set for a correspondence $S\subset E^2$ 
if $p_1^{-1}(Z) = p_2^{-1}(Z)$. By \cite[Theorem 1]{_Bellaiche_}, if 
for a curve $S\subset E^2$ there exist infinitely many irreducible 
complete sets $Z\subset E$, then there exists a morphism $f\colon E\to \p^1$ 
such that $f\circ p_1 =  f\circ p_2\colon S\to \p^1$; that is, $S$ is a 
quasi-diagonal for the line bundle $L = f^*\OP_{\p^1}(1)$. 
On the other hand, if the number of irreducible complete sets is 
small, the self-correspondence exhibits trivial dynamics in view 
of \cite[Theorem 2]{_Bellaiche_}.

%%%%%%%%%%%%%%%%%%%%%%%%%%%%%%%%%%%%%%%%%%%%%%%%%%%%%%%%%%%%
\subsection{Quasi-diagonals and complex geometry of BG-precursor}
%%%%%%%%%%%%%%%%%%%%%%%%%%%%%%%%%%%%%%%%%%%%%%%%%%%%%%%%%%%%

Surprisingly, we were able to apply our argument
in the converse direction, deducing theorems about
quasi-diagonals from results on complex geometry
of Bogomolov-Guan precursors. 

Let $X \to E^n$ be a BG-precursor.
It is not hard to see that as well as the Kodaira surface, 
the corresponding BG-precursor is
 a nilmanifold (see Subsection
 \ref{_Nilmanifold_Subsection_}),
that is, a quotient of a nilpotent Lie group
by a cocompact, discrete subgroup acting by translations.
This nilpotent group is not a complex Lie group.
However, it carries an invariant complex structure,
obtained from a complex structure operator on 
its Lie algebra. The integrability of this complex
structure is interpreted
as a Lie algebra condition
(\ref{_integra_comple_on_Lie_Definition_}).

A Hermitian metric $\omega$ 
on a complex manifold is called {\bf pluriclosed},
or {\bf SKT}, if it satisfies $dd^c\omega=0$.

Special metrics on nilmanifolds are related
to 2-forms on their Lie algebra. If we identify the space
of invariant forms on a nilmanifold with the Grassmann
algebra of its Lie algebra $\g$, de Rham differential
is induced from a map $\Lambda^1(\g) \to \Lambda^2(\g)$
which is dual to the Lie bracket. Knowing an
explicit expression for the Lie bracket and
the complex structure, it is not hard to show
that a given 2-form is $dd^c$-closed.

Using this construction, we 
obtain a pluriclosed form on any 
BG-precursor manifold (\ref{_precursor_SKT_Theorem_}).

An argument based on Gromov's compactness
(\cite{_Verbitsky:rat-twi_}) implies then
that each component in the Barlet 
moduli space of complex curves on a pluriclosed manifold
$X$ is compact (\ref{_SKT_then_Barlet_compact_Corollary_}).
This was used to show that all families of quasi-diagonals
on a product of elliptic curves are discrete; in particular, there
are countably many quasi-diagonals 
(\ref{_quasidiag_countable_via_SKT_Theorem_}).
An algebraic proof of this result is given in
\ref{_theorem_qd_}.

%%%%%%%%%%%%%%%%%%%%%%%%%%%%%%%%%%%%%%%%%%%%%%%%%%%%%%%%%%%%
\subsection{Main results}
%%%%%%%%%%%%%%%%%%%%%%%%%%%%%%%%%%%%%%%%%%%%%%%%%%%%%%%%%%%%
Our first result describes projective and Fujiki class C subvarieties in a BG-manifold and its 
precursor.
Here we use notation introduced in 
Section \ref{_subsection_BG_precursor_}.

\hfill

\theorem\label{_theorem_BG_mfld_and_precursor_}
Let $Y$ be a subvariety of a Bogomolov-Guan manifold $BG$ 
(or of its precursor $X$).
Set $Z=\alpha^{-1}(P(Y))$ (or $Z = P_X(Y)$, respectively). 
Then $Y$ has the following properties:
\begin{enumerate}
    \item[(1)] $Y$ is projective if and only if all components 
    of $Z$ are points or quasi-diagonals\footnote{Restricted quasi-diagonals
for the BG-manifold and the restricted precursor, any quasi-diagonals for the
unrestricted precursor.} over $(E,L)$ where
    $E$ and $L$ are associated with 
    the Kodaira surface $\Kod$ (see \ref{_definition_Kodaira_}).
    \item[(2)] $Y$ is in Fujiki class C if and only if each component 
    $S$ of $Z$ is either a point or a curve such that the degrees of 
    the projections $p_i\colon S\to E$ to the components of the product are the 
    same for all $1\leqslant i\leqslant n+1$.
\end{enumerate}

\hfill

This theorem, as well as the methods we used in its proof (see 
\ref{_BKKY_main_Theorem_}, 
\ref{_Proof_for_class_C_Theorem_},
\ref{_proj_then_quasidiag_Theorem_}) are very similar 
to \cite[Theorem B]{_BKKY_}. However, by analyzing the construction in more 
detail we have slightly improved the result. In particular, we showed that 
a Moishezon subvariety of a Bogomolov-Guan manifold is necessarily 
projective.  Also we show that any smooth projective
subvariety in a BG-precursor is necessarily K\"ahler
(\ref{_Proof_for_class_C_Theorem_}).

In order to understand better the projective subvarieties in 
Bogomolov-Guan manifolds and precursors, we study quasi-diagonals.
Our second result shows that quasi-diagonals never come in continuous 
families. 

\hfill

\theorem\label{_theorem_qd_}
Let $E$ be an elliptic curve and $L$ be an ample line bundle on $E$. 
Then, for any integer $n \geqslant 2$, 
the product $E^n$ contains at most a countable set of
quasi-diagonals over $(E,L)$.

\hfill

We provide two different proofs of this theorem
using distinct methods. In both proofs the first step is to reduce 
to the case $n=2$, see \ref{_quasi-diagonals_in_En_Theorem_}.

The first proof in Section \ref{_section_families_of_qd_}
uses purely algebro-geometric methods: we assume that there exists a family
of quasi-diagonals over $(E,L)$ in $E^2$ parametrized by a variety $T$, 
consider the universal space $\pi\colon X\to T$ of the family and 
study the inverse image of the 
bundle $M = p_1^*L\otimes p_2^*L^{-1}$ to $X$. 
On the one hand, the property of quasi-diagonals 
implies that $M$ equals a pull-back of a line bundle on $T$. 
On the other hand, we show that the inverse image of the 
line bundle $p_i^*L$ to $X$ is supported on a multisection of $\pi$. 
This leads us to a contradiction.

The second proof in Section \ref{_Countability_coreduction_Subsection_}
uses Campana's coreduction theorem. This theorem 
states that if $X$ is a sufficiently good compact complex variety
then there exists a meromorphic map $X\to Y$ which contracts 
almost all Moishezon subvarieties in $X$. If $E^2$ contains a 
continuous family of quasi-diagonals then the variety $X = \Kod^2/E_1$
satisfies the coreduction theorem and two general points of $X$
are connected by a chain of projective varieties similarly 
to \ref{_theorem_BG_mfld_and_precursor_}. Thus, the coreduction 
map contracts $X$ to a point, which is impossible since $X$
is not Moishezon (a Kodaira surface $\Kod$ can be embedded in $X$).

\ref{_theorem_qd_} demonstrates that quasi-diagonals are quite rare.
Thus, one can ask whether there exists a quasi-diagonal or restricted quasi-diagonal
over $(E,L)$ for a given elliptic curve $E$ and a line bundle $L\in \Pic(E)$.
We construct several examples which show that in some cases we get 
large sets of (restricted) quasi-diagonals while generally there are no 
restricted quasi-diagonals.

\hfill

\proposition\label{_proposition_qd_}
Let $E$ be an elliptic curve and $L$ be an ample line bundle on $E$. 
Then the following assertions hold.
\begin{enumerate}
    \item[(1)] For each elliptic curve $E$, a line bundle $L\in \Pic_d(E)$ for 
    $d\geqslant 1$ and integer $n\geqslant 2$ there is a countable set of 
    quasi-diagonals over $(E,L)$ in $E^n$.
    \item[(2)] If $n$ is even or if the $j$-invariant of $E$ equals $0$
    then there exists $L\in \Pic_d(E)$ where $d\geqslant 1$ such that 
    there is a countable set of restricted quasi-diagonals over $(E,L)$ in $E^n$.
    \item[(3)] For each elliptic curve $E$ and integer $n\geqslant 2$ the set 
    of line bundles $L\in \Pic_d(E)$ for $d\geqslant 1$ such that $E^n$ admits
    no restricted quasi-diagonal over $(E,L)$ is a complement 
    of a countable subset in $\Pic_d(E)$.
\end{enumerate}

\hfill

The first and the second assertions of \ref{_proposition_qd_}
are proved with concrete constructions provided in Section 
\ref{_section_examples_of_qd_}. The third assertion 
follows from \ref{_theorem_qd_},
see \ref{_corollary:_line_bundles_for_restricted_qd_Corollary}.

As an immediate consequence of \ref{_theorem_BG_mfld_and_precursor_} 
and \ref{_proposition_qd_}, 
we obtain a generic property of Bogomolov-Guan manifolds. 

\hfill

\corollary
 Let $\Kod$ be a Kodaira surface associated to an elliptic curve 
 $E$ and a line bundle $L\in \Pic_d(E)$ as in \ref{_definition_Kodaira_}.
 If $L\in \Pic_d(E)$ is very general then all projective subvarieties in 
 a Bogomolov-Guan manifold or its restricted precursor obtained from $\Kod$ 
 are strictly contained within the fibers of its Lagrangian fibration.

%%%%%%%%%%%%%%%%%%%%%%%%%%%%%%%%%%%%%%%%%%%%%%%%
\subsection{Structure of this paper}
%%%%%%%%%%%%%%%%%%%%%%%%%%%%%%%%%%%%%%%%%%%%%%%%

In Section \ref{_section_qd}, we formally define
quasi-diagonals and restricted quasi-diagonals, provide
several explicit examples (including symmetric curves and
non-elliptic quasi-diagonals), and reduce the study of
quasi-diagonals in $E^n$ to $E^2$. The section concludes
with the algebraic proof of \ref{_theorem_qd_} and states
several open questions.

In Section \ref{_section_bg_manifolds}, we review the
theory of principal toric and elliptic fibrations,
describe the construction of Kodaira surfaces as
nilmanifolds, and detail the construction of
Bogomolov-Guan manifolds and their precursors. Using the
theory of principal toric bundles, we then prove the
classification theorems for projective and Fujiki class C
subvarieties in the Bogomolov-Guan manifolds.

In Section \ref{_section_bg_precursors}, we prove the
analogous classification result for subvarieties in the
Bogomolov-Guan precursor, completing the proof of
\ref{_theorem_BG_mfld_and_precursor_}.

Section \ref{_section_Barlet_spaces_} provides a
complex-analytic perspective on our rigidity results. We
recall the definitions of Barlet spaces, Campana's
coreduction map, and SKT (pluriclosed) metrics. We show
that the Bogomolov-Guan precursor admits an SKT
metric. Using Gromov's compactness theorem for Barlet
spaces, we provide an independent geometric proof of
\ref{_theorem_qd_}.

Finally, Appendix \ref{section: no quasi-diagonals over
  lines} presents a detailed analysis of the preimages of
general lines in a $4$-dimensional Bogomolov-Guan
manifold, demonstrating that they do not yield
quasi-diagonals in general and further illustrating the
scarcity of restricted quasi-diagonals.

%%%%%%%%%%%%%%%%%%%%%%%%%%%%%%%%%%%%%%%%%%%%%%%%%%%%%%%%%%%%%%%%%%%%%%%%%%%%

\section{Quasi-diagonals in a product of elliptic curves}\label{_section_qd}

%%%%%%%%%%%%%%%%%%%%%%%%%%%%%%%%%%%%%%%%%%%%%%%%%%%%%%%%%%%%%%%%%%%%%%%%%%%%

%%%%%%%%%%%%%%%%%%%%%%%%%%%%%%%%%%%%%%%%%%%%%%%%%%%%%%%%%%%%%%%%%%%%%%%%
\subsection{Quasi-diagonals: definition and basic properties}\label{_section_qd_definition_}
%%%%%%%%%%%%%%%%%%%%%%%%%%%%%%%%%%%%%%%%%%%%%%%%%%%%%%%%%%%%%%%%%%%%%%%%

\remark
In this paper we assume that {\bf an elliptic curve}
has a fixed point and a group operation, and
{\bf a 1-dimensional complex torus} is 
a 1-dimensional smooth complex curve of genus 1.
This distinction will be important in the sequel.

\hfill

\definition
The group of line bundles on a complex manifold $X$ is denoted by $\Pic(X)$,
and is called {\bf the Picard group}. When $X$ is a compact complex curve,
we have $\deg(L)=\int_X c_1(L)$. Then the Picard group is a union
of $\Z$ connected components $\Pic_i(X)$, with
$\Pic_i(X)$ denoting the space of line bundles of degree $i$.
Each $\Pic_i(X)$ is a torsor over the group $\Pic_0(X)$, which is
called {\bf the Jacobian of $X$}. The Jacobian is
a complex torus, isomorphic to $E$ when $X=E$ is an elliptic curve.

\hfill

\remark
Let $E$ be an elliptic curve. Then 
the standard action of $\Pic_0(E)$ on $\Pic(E)$
 coincides with the action induced by the action of $E$ on itself
by translations.

\hfill

\definition
Let $L$ be an ample line bundle on an elliptic curve $E$.
Let $p_i:\; E^n \arrow E$ denote the projection to the $i$-th component.
{\bf A quasi-diagonal} over $(E,L)$ is a curve $S\subset E^n$
such that for all $i, j$, the bundle $V_{i,j}:=p_i^*L|_S\otimes (p_j^*L|_S)^{-1}$
is torsion, that is, satisfies $V_{i,j}^{\otimes k}=\calo_S$
for some $k>0$.

\hfill

\remark
The degree of the bundle $p_i^*L|_S$ is equal to $\deg L \cdot d_i$,
where $d_i$ is the degree of the restriction $p_i:\; S \arrow E$.
Therefore,  the map $p_i:\; S \arrow E$
 has the same degree for all $i$ if $S$ is a quasi-diagonal.

\hfill

One of the aims of the present paper is the following
theorem. 

\hfill

\theorem
There are at most countably many quasi-diagonals for any
given $(E,L)$.

\proof \ref{_at_most_countably_many_quasi-diagonals_Theorem_}. \endproof

\hfill

\remark
The notions we discussed in this section so far are
``identity agnostic'': they are 
defined for a 1-dimensional complex torus, without
fixing the unity. The notions which are identity agnostic
are functorial with respect to the action of an elliptic
curve on itself by parallel transport.
The definition of 
``restricted quasi-diagonal'' in the next paragraph is no longer identity
agnostic: it is defined in terms of a group structure
on $E$ and requires one to fix the identity element.

\hfill

\definition
A {\bf restricted quasi-diagonal} is a 
quasi-diagonal $S\subset E^n$ such that the natural summation
map $E^n \to E$, $(x_1,..., x_n)\to \sum x_i$ takes $S$ to a point.

\hfill

\remark 
For any elliptic curve $E$,  
a line bundle $L$ on $E$ and an integer $n\geqslant 2$ 
the pair $(E,L)$ admits a countable set of quasi-diagonals in $E^n$
(\ref{_quasi-diagonals_from_automorphisms_of_E_Example_}).
The restricted quasi-diagonals are, by contrast,
rather rare, and do not exist for general $(E,L)$, 
by \ref{_corollary:_line_bundles_for_restricted_qd_Corollary}. 
Both types of quasi-diagonals correspond
to algebraic subvarieties in certain non-K\"ahler
manifolds (see \ref{_BKKY_main_Theorem_} and \ref{_proj_then_quasidiag_Theorem_}).

\hfill

\remark
In Subsection~\mbox{\ref{_qd_reduced_to_E^2_Subsection_}} we are going to reduce the classification of quasi-diagonals in $E^n$
to the classification of quasi-diagonals in $E^2$.

%%%%%%%%%%%%%%%%%%%%%%%%%%%%%%%%%%%%%%%%%%%%%%%%%%%%%%%%%%%%%%%%%%%%%%%%
\subsection{Examples of quasi-diagonals and restricted quasi-diagonals}\label{_section_examples_of_qd_}
%%%%%%%%%%%%%%%%%%%%%%%%%%%%%%%%%%%%%%%%%%%%%%%%%%%%%%%%%%%%%%%%%%%%%%%%

\example \label{_quasi-diagonals_from_automorphisms_of_E_Example_}
Let $\nu:\; E \arrow E$ be an automorphism of order $d$.
Then $\widetilde{L}:=\prod_{i=0}^{d-1} (\nu^i)^* L_1$ is $\nu$-invariant, for any given $L_1$, hence
the graph $\Gamma_\nu\subset E^2$ of $\nu$ is a quasi-diagonal for $(E,\widetilde{L})$.
Note that the map $\Pic_k(E) \to \Pic_{dk}(E)$ given by $L \mapsto L^{\otimes d}$ is surjective. 
Thus, there exists a line bundle $L$ such that $L^{\otimes d} = \widetilde{L}$ and $\Gamma_{\nu}$ is 
a quasi-diagonal for $(E,L)$ as well. 

Since $\Pic_0(E)$ acts transitively on $\Pic_k(E)$,
this also implies the existence of a countable set of quasi-diagonals in $E^d$
for all $(E,L)$.

%The group $\Aut(E)$ contains all translation 
%by torsion elements; therefore, for any elliptic curve $E$ 
%and any positive integers $k$ and $d$ there exists 
%a line bundle $L$ in $\Pic_k(E)$
%and a quasi-diagonal in $E^d$ for $(E,L)$. 

\hfill

\example \label{example: restricted quasi-diagonal for even} 
Here we construct a restricted 
quasi-diagonal in $E^{2n}$. Let $\nu = -\mathrm{id}_E$. 
Then the graph of $\nu$ is a restricted quasi-diagonal 
over $(E,\OP_E(e))$
where $e$ is the zero point in $E$. Taking the
fiber products of several quasi-diagonals of this  type 
as in \ref{_quasi-diagonals_in_En_Theorem_} below,
we can construct a restricted quasi-diagonal in $E^{2n}$
for any integer $n\geqslant 1$.

\hfill

%%%%%%%%%%%%%%%%%%%%%%%%%%%%%%%%%%%%%%%%%%%%%%%%%%%%%%%%%%%%%%%%%%%%%%%%
\example \label{_restricted_quasi-diagonal_for_3_Example_} 
Here we construct a restricted quasi-diagonal in $E^3$ for 
an elliptic curve $E$  whose $j$-invariant equals $0$.
In this case
$E =\C/ \mathbb{Z}[\zeta]$ where $\zeta$ is a primitive 3rd degree root of unity. 
Denote by $e$ the zero on the elliptic curve $E$ and let $S$ be the image of 
the following embedding to $E^3$:
\begin{align*}
 \left(\mathrm{id}_E\times\zeta\times \zeta^2\right)\colon E \to E^3; && x\mapsto (x,\zeta x, \zeta^2 x).
\end{align*}
Then $S$ is a restricted quasi-diagonal for the line bundle $\OP_E(e)$. 
Taking fiber product of quasi-diagonals in \ref{example: restricted quasi-diagonal for even}
and in \ref{_restricted_quasi-diagonal_for_3_Example_} we can construct 
a restricted quasi-diagonal in $E^n$ for $E  =\C/ \mathbb{Z}[\zeta]$ and any integer $n\geqslant 5$.

\hfill

\example \label{example: non-elliptic quasi-diagonal} 
Here we give an example of a quasi-diagonal which is not an elliptic curve.
Let $S = \{(x:y:z)\in \p^2\mid x^4+y^4+z^4 = 0\}$ be a Fermat quartic curve. 
Denote by $\sigma$ and $\phi$ the following automorphisms of $S$:
\begin{align*}
 \sigma(x:y:z) = (z:x:y); && \phi(x:y:z) = (-x:y:z).
\end{align*}
The orders of automorphisms $\sigma $ and $\phi$ equal $3$
and $2$ respectively. Note also that $\sigma$ and $\phi$
do not commute. By the Hurwitz formula  one can check that the
curve $E = S/\langle \sigma\rangle$ is elliptic.
Denote by $p$ the factorization map $p\colon S\to S/\langle \sigma\rangle$ by the group $\langle \sigma\rangle = \mathbb{Z}/3\mathbb{Z}$.
 Consider then the following map:
\[
 f = (p,p\circ \phi)\colon S\to E^2
\]
The map $f$ is not injective: for example, $f((1:0:\zeta)) = f((-\zeta:1:0))$ where $\zeta$ is a degree 8 primitive root of unity.
Nevertheless, let us check that the map $f$ is generically 1 to 1 onto its image. Assume that the distinct points $s_1$ and $s_2$ map to the same point under $f$. This implies that
\begin{equation}\label{_projections_endom_Equation_}
  p(s_1) = p(s_2); \ \ \ \  p(\phi(s_1)) = p(\phi(s_2)).
\end{equation}
By construction, $p(s_1)=p(s_2)$ implies~$s_2 = \sigma^{k}
s_1$ for $k=\pm 1$. The second equation of
\eqref{_projections_endom_Equation_} implies that
$\phi(\sigma^{k} (s_1)) = \sigma^l(\phi(s_1))$  for $l=\pm 1$. Thus, $s_1$ and $s_2$ map to the same point under $f$ if and only if $s_2 = \sigma^{k} s_1$ and $s_1$ is a fixed point of the automorphism $g = \phi\circ \sigma^{-l}\circ \phi\circ \sigma^k$ where $k,l = \pm 1$. An easy computation shows that for any choice of $k$ and $l$ the automorphism $g$ is not the identity on $S$. Thus, $f$ is generically $1$ to $1$.

This shows that $f$ is a birational map to its image;
thus, $f(S)$ is a singular non-elliptic curve of geometric
genus $3$. By \ref{_symmetric_curves_are_quasi-diagonals_Lemma_}
below, the image is a quasi-diagonal
for a suitable ample line bundle on the curve $E$.

\hfill

\remark \label{_remark_translations_of_qd_by_torsion_}
Let  $\varphi = (\varphi_1,\dots,\varphi_n) 
\colon S\hookrightarrow E^n$ be an embedding of a quasi-diagonal 
over $(E,L)$, where $L$ is an ample bundle. Fix the
unity and a 
set $t_1,\dots,t_n$ of $k$-torsion elements in $E$. Consider the following 
new embedding of $S$ to $E^n$:
\begin{align*}
\varphi_{t_1,\dots,t_n}\colon S\hookrightarrow E^n; && \varphi_{t_1,\dots,t_n}(x) = (\varphi_1(x)+t_1,\dots,\varphi_n(x)+t_n).
\end{align*}
Then $\varphi_{t_1,\dots,t_n}(S)\subset E^n$ is a quasi-diagonal 
over $(E,L)$. Indeed, for any divisor $D\subset E$ of
degree $l$ associated with the line bundle $B=O(D)$,  
the translation $\tau_{t_i}$ by $t_i$ takes $B$ to
$\tau_{t_i}^*(B)=\calo(D+l(t_i-e))$, which implies $B^k=\tau_{t_i}^*(B)^k$.
Applying this argument to $L$, we obtain that
$\varphi_{t_1,\dots,t_n}^*(L^k)= \varphi^*(L^k)$.
Therefore, if $E^n$ contains a quasi-diagonal 
then it contains a countable set of quasi-diagonals.
By the same reason, once $E^n$
contains a restricted quasi-diagonal, then it contains a
countable set of them: the translation by
$t_1,\dots,t_n$ preserves the restricted
quasi-diagonals if $\sum t_i=0$.

\hfill

%%%%%%%%%%%%%%%%%%%%%%%%%%%%%%%%%%%%%%%%%%%%%%%%%%%%%%%%%%%%%%%%%%%%%%%%
\subsection{Quasi-diagonals in $E^n$ and in $E^2$}
\label{_qd_reduced_to_E^2_Subsection_}
%%%%%%%%%%%%%%%%%%%%%%%%%%%%%%%%%%%%%%%%%%%%%%%%%%%%%%%%%%%%%%%%%%%%%%%%

\remark 
Let $\Pi:\; E^n \arrow E^m$ be a projection of $E^n$
onto $m$ of its factors, $m\leq n$. Clearly, $\Pi(S)$ is a quasi-diagonal for
any quasi-diagonal $S\subset E^n$. This implies, in
particular, the following claim:

\hfill

%%%%%%%%%%%%%%%%%%%%%%%%%%%%%%%%%%%%%%%%%%%%%%%%%%%%%%%%%%%%
\remark
Consider the projection $p_{i,j}:\; E^n \arrow E^2$,
$p_{i,j}(x_1, ..., x_n)=(x_i, x_j)$. Then,  for any quasi-diagonal
$S\subset E^n$, its projection $p_{i,j}(S)$
is a quasi-diagonal in $E^2$ for any $i, j$. Conversely,
 $S\subset E^n$ is a quasi-diagonal if $p_{i,j}(S)$ is a quasi-diagonal
for any $i,j$.

\hfill

Taking fibered products of quasi-diagonals in $E^2$, 
we can obtain all quasi-diagonals in $E^n$ (but {\em not}
restricted quasi-diagonals).

\hfill

\definition
Let $f:\; X\arrow Z$, $g:\; Y\arrow Z$ be any maps.
{\bf The fibered product} $X\times_Z Y$
is the set $\{(x,y)\in X\times Y\ \ |\ \ f(x)=g(y)\}$.

\hfill

%%%%%%%%%%%%%%%%%%%%%%%%%%%%%%%%%%%%%%%%%%%%%%%%%%%%%%%%%%%%
\theorem \label{_quasi-diagonals_in_En_Theorem_}
Let $S_1, ..., S_{n-1}$ be quasi-diagonals in $E^2$ for the same line bundle $L$ on $E$.
Consider the fibered product
$E^n= (E^2)\times_E (E^2)\times_E (E^2)\times_E ... \times_E (E^2)$.
where the projection of $E^2$ to the left factor is $(x,y)\arrow x$
and to the right is $(x,y)\arrow y$. 
Then $S = S_1\times_E S_2 \times_E ... \times_E S_{n-1}$
is a quasi-diagonal, and all quasi-diagonals are obtained this way.

\hfill

\proof Denote by $p_{i,j}\colon S_i\to E$ the restriction
to the curve $S_i$ of the projection~\mbox{$E^2\to E$} to
the $j$-th component of the product for $1\leqslant
i\leqslant n-1$ and $j = 1$ or $2$. Since~$S_i$ is a
quasi-diagonal,  one has $p_{i,1}^*L \otimes
\left(p_{i,2}^*L \right)^{-1} =: \M_i$ is torsion. Denote
by  $q_i\colon S\to S_i$ the projection to the $i$-th
component of the fiber product. By construction 
$p_i = p_{i,1}\circ q_i = p_{i-1,2}\circ q_{i-1}\colon
S\to E$ is the restriction to $S$ of the projection
$E^n\to E$ to the $i$-th component. Thus, for all
$2\leqslant i\leqslant n$ one has the following equality:
\begin{multline*}
p_{i-1}^*L = q_{i-1}^*(p_{i-1,1}^*L) = q_{i-1}^*(p_{i-1,2}^*L\otimes \mathcal{M}_{i-1}) = \\ = q_i^*(p_{i,1}^*L)\otimes q_{i-1}^*(\M_{i-1}) = p_i^*L\otimes q_{i-1}^*(\M_{i-1}).
\end{multline*}
Since $\M_{i-1}$ is torsion by construction, 
we deduce that $S$ is a quasi-diagonal.

Now if $S \subset E^n$ is a quasi-diagonal  and $p_i\colon S\to E$ is the projection to the $i$-th component then we define $S_i = (p_i,p_{i+1})(S)\subset E^2$ for all $1\leqslant i\leqslant n-1$. One can easily check that $S_i$ is a quasi-diagonal and $S$ is isomorphic to the corresponding fiber product.
\endproof

\subsection{Symmetric curves in $E^n$ and quasi-diagonals} \label{section: lines don't give restricted qd}

We denote the
symmetric group of $n$-tuples by $\SG_n$.
If $S\subset E^n$ is a quasi-diagonal then the degrees of
projections $p_i\colon S\to E$ are the same for all
$1\leqslant i\leqslant n$.  Thus, it is natural to look for
quasi-diagonals among $\SG_n$-symmetric curves in $E^n$. 
\ref{_symmetric_curves_are_quasi-diagonals_Lemma_} 
shows that under an additional assumption (which is rather strong in fact) a symmetric curve is indeed a quasi-diagonal. 

\hfill

\lemma \label{_symmetric_curves_are_quasi-diagonals_Lemma_}
Let $S\subset E^n$ be a $\SG_n$-invariant curve.
Assume that the projection map $p_1\colon S\to E$ is a Galois
cover, that is there exists a normal subgroup $H\subset
\Aut(S)$ such that $p_1$ is the factorization map 
$S\to S/H\cong E$. Then there exists an ample 
line bundle $L$ such that $S$ is a
quasi-diagonal over $(E,L)$.

\hfill

\proof Denote by $p_i\colon S\to E$ the restriction of the
projection $E^n\to E$ on the $i$-th component to $S$. 
Let $\sigma_i\in \SG_n$ be a transposition $(1,i)$. Since $S$ is 
$\SG_n$-invariant, $\sigma_i$ induces an automorphism of $S$ 
and one has $p_i = p_1\circ \sigma_i$.
 
Let $D$ be an $\Aut(S)$-invariant
subset of points on $S$. Since $\SG_n$ acts
on $S$, $D$ is an $\SG_n$-invariant subset on $E^n$.
Since $D$ is $H$-invariant, we get
$\OP_S(D) = p_1^*L$, where $L = \OP_E(p_1(D))$. On the
other hand, since $\sigma_i^*D = D$, we also observe that
$\OP_S(D) = p_i^*L$ for all $1\leqslant i\leqslant n$. 
\endproof

\hfill

\remark
Recall that $\sigma\colon E^3 \to E$ is the map
$(x,y,z)\mapsto x+y+z$. Denote by $\alpha\colon E^3\to
E^{(3)}$ the symmetrization map. Note that $E^{(3)}$ is
isomorphic to a $\p^2$-bundle over $E$ and $\sigma$
commutes with the $\p^2$-bundle structure. 
The quasi-diagonal described in 
\ref{_restricted_quasi-diagonal_for_3_Example_} lies in
the kernel of $\sigma$ (hence it is a restricted
quasi-diagonal). By \ref{_lemma_restricted_qd_in_E3_maps_to_line_} 
its image under $\alpha$ is a
line in $\p^2$. Note that the preimage of this line 
under $\alpha$ is reducible: it consists of $S$ and $\tau(S)$
where $\tau$ is a transposition in $\SG_3$.

Inspired by this fact, we study preimages of lines in
$\p^2$. In the appendix \ref{section: no quasi-diagonals over lines} 
we show that the preimage of a general line in $\p^2$ is a
smooth $\SG_3$-invariant curve in $E^3$. However, 
once the preimage of a line in $E^3$ is smooth, it is not
a quasi-diagonal for any line bundle $L$ on $E$
(\ref{_proposition_S_l_is_not_quasi-diagonal_Lemma_}).

\hfill

%%%%%%%%%%%%%%%%%%%%%%%%%%%%%%%%%%%%%%%%%%%%%%%%%%%%%%%%%%%%%%%%%%%%%%%%
\subsection{Families of quasi-diagonals}\label{_section_families_of_qd_}
%%%%%%%%%%%%%%%%%%%%%%%%%%%%%%%%%%%%%%%%%%%%%%%%%%%%%%%%%%%%%%%%%%%%%%%%

The main result of this subsection is the following theorem.

\hfill

%%%%%%%%%%%%%%%%%%%%%%%%%%%%%%%%%%%%%%%%%%%%%%%%%%%%%%%%%%%%%%%%%%%%%%%%
\theorem \label{_at_most_countably_many_quasi-diagonals_Theorem_}
Let $E$ be an elliptic curve and $L$ an ample line bundle on $E$.
Then $E^n$ contains at most countably many quasi-diagonals over $(E,L)$.

\hfill

We prove this result at the end of this
subsection. Before,
we notice that
\ref{_at_most_countably_many_quasi-diagonals_Theorem_}
has an immediate corollary.

\hfill

%%%%%%%%%%%%%%%%%%%%%%%%%%%%%%%%%%%%%%%%%%%%%%%%%%%%%%%%%%%%%%%%%%%%%%%%
\corollary \label{_corollary:_line_bundles_for_restricted_qd_Corollary}
Let $E$ be an elliptic curve.
Then the set of line bundles $L\in \Pic_m(E)$ such that there exists a restricted quasi-diagonal in $E^n$ for $L$
is at most countable.

\hfill

\pstep
The group $E$ acts on itself by
translations; we denote an automorphism
associated with $x\in E$ by $\tau_x$. Clearly, for any element $x\in E$,
this action takes quasi-diagonals $S \subset  E^n$
over $(E, L)$ to quasi-diagonals over $(E,\tau_{-x}^*L)$.

\hfill

{\bf Step 2:}
Fix an ample line bundle $L\in \Pic_m(E)$ over $E$.
Let $z\in E$ and let ${\goth Q}_z(E, L)$
denote the set of quasi-diagonals $S\subset  E^n$ over $(E,L)$ which satisfy
$\sigma(s) =z$, where $\sigma:\; E^n\to E$ is the
summation map.
Clearly, $\tau_x({\goth Q}_z(E, L))={\goth Q}_{mx}(E, \tau_{-x}^*L)$.
Since $E$ acts on $\Pic_m(E)$ transitively, for each
$L'\in \Pic_m(E)$ there exists $x\in E$ such that
$\tau_{-x}^*L'=L$. Therefore, for any
restricted quasi-diagonal $S\subset  E^n$ over $(E, L_1)$,
the curve $\tau_xS\subset E^n$ is a quasi-diagonal over
$(E,L)$, and satisfies $S\in {\goth Q}_{mx}(E, L)$.
If the set of restricted quasi-diagonals
is non-empty for noncountably many $L\in \Pic_m(E)$,
there is a quasi-diagonal $S\in {\goth Q}_{mx}(E, L)$
for non-countably many points $mx\in E$.
This implies that the set of quasi-diagonals over $(E,L)$
is non-countable, contradicting 
\ref{_at_most_countably_many_quasi-diagonals_Theorem_}.
\endproof

\hfill

In order to prove
\ref{_at_most_countably_many_quasi-diagonals_Theorem_} we
need the following variation of the rigidity lemma
(\cite[Proposition 5.3]{_Kollar_}).

\hfill

\begin{lemma}\label{_lemma:_map_to_curve_factors_Lemma_}
  Let $f\colon X\to Y$ and $g\colon X\to Z$ be surjective morphisms between smooth varieties and assume that $f$ is proper. Assume moreover that $Z$ is a projective curve, that a general fiber of $f$ is connected and that there exists $y_0\in Y$ such that $g(f^{-1}(y_0))$ is a point in $Z$. Then there exists a rational map $h\colon Y\dashrightarrow Z$ such that $g = h\circ f$.
 \end{lemma}

 \hfill
 
 \begin{proof}  
  Fix a line bundle $\L$ on $Z$ such that $h^0(Z,\L) \geqslant 2$.
  For a general point $y\in Y$ denote by $X_y$ the fiber 
$f^{-1}(y)\subset X$. Consider the following function on $Y$:
  \[
   y\mapsto h^0\left(X_y, (g^*\L)|_{X_y}\right).
  \]
  Since $g|_{X_{y_0}}$ maps $X_{y_0}$ to a point, then
  $(g^*\L)|_{X_y}$ equals $\OP_{X_y}$. By the
  semi-continuity theorem we deduce that 
$h^0\left(X_y, (g^*\L)|_{X_y}\right) \leqslant 1$ 
for a general point $y\in Y$. On the other hand, if $g(X_y) = Z$ then $h^0\left(X_y, (g^*\L)|_{X_y}\right)$ is at least $h^0(Z,\L) \geqslant 2$. Therefore, there exists an open subset $U\subset Y$ such that for any $y\in U$ one has that~$g(X_y)$ is a point. This correctly defines the rational map $h$ which is regular on the open subset $U$.
 \end{proof}

\hfill
 
 Let $E$ be an elliptic curve and let $\L$ be an ample line bundle on $E$. Denote by $\pr_i\colon E^2\to E$ the projection to the $i$-th component of the direct product, where $i = 1$ or $2$. 

\hfill

  \begin{lemma}\label{lemma: x-y is torsion equivalently the line bundle is torsion}
  Let $x$ and $y$ be points on an elliptic curve $E$. The line bundle~$\OP_E(x-y)$ is a torsion if and only if the difference $x-y$ is a torsion on~$E$.
 \end{lemma}

\hfill
 
 \begin{proof}
  One can assume that $y = e$ is the zero point on the elliptic curve~$E$. There is the abelian variety isomorphism between the elliptic curve $E$ and its dual:
  \begin{align*}
   \varphi\colon E\xrightarrow{\cong} \Pic_0(E); && \varphi(x) = \OP_{E}(x - e).
  \end{align*}
  Thus, the image of the point $x\in E$ under the map $\varphi$ is a torsion line bundle  if and only if it is a torsion point on $E$.
 \end{proof}

\hfill
 
 \begin{lemma}\label{lemma: F1 - F2 is not torsion}
  Let $e$ be a point on $E$. Then $ \OP_{E^2}(\pr_1^{-1}(e) - \pr_2^{-1}(e))$ is not a torsion line bundle on $E^2$.
 \end{lemma}

 \hfill
 
 \begin{proof}
   Let $\varphi\colon E\to E^2$ be the map $\varphi(x) = (x,x+t)$ where $t\in E$ is not a torsion point. Then $\varphi^*(\OP_{E^2}(\pr_1^{-1}(e) - \pr_2^{-1}(e)))\cong \OP_E([e] - [e-t])$ and this line bundle cannot be torsion by \ref{lemma: x-y is torsion equivalently the line bundle is torsion}.
 \end{proof}

\hfill
 
 \begin{lemma}\label{lemma: pullback is torsion}
  Let $\alpha\colon X\to E^2$ be a surjective morphism, where $X$ is a projective variety. If $\alpha^*\L$ is a torsion line bundle then so is the line bundle $\L\in \Pic_0(E^2)$.
 \end{lemma}

\hfill

 \begin{proof}
  The map $\alpha$ factors through the Albanese variety of
  $X$ which maps surjectively to $E^2$. Thus, the inverse
  image map $\alpha^*\colon \Pic_0(E^2)\to \Pic_0(X)$ is
  an isogeny to its image. In particular, the kernel of
  the map $\alpha^*$ is finite; this implies the
  assertion.
 \end{proof}

\hfill

%%%%%%%%%%%%%%%%%%%%%%%%%%%%%%%%%%%%%%%%%%%%%%%%%%%%%%%%%%%%
\remark\label{_pullback_vanishes_when_torsion_Remark_}
The same argument works to prove the following
theorem: let $f:\; X \to Y$ be a dominant map of 
projective varieties (or varieties of Fujiki class C).
Then the pullback of a non-torsion line bundle is
never torsion. Indeed, $f$ induces an injective map 
$H^{0,1}(Y)\to H^{0,1}(X)$, hence $f^*$ induces a map on Picard
varieties $\Pic(Y) \to \Pic(X)$ which has finite kernel.

\hfill

 Now we are ready to prove \ref{_at_most_countably_many_quasi-diagonals_Theorem_}.

\hfill

 \noindent{\bf Proof of \ref{_at_most_countably_many_quasi-diagonals_Theorem_}.}
  {\bf Step 1:} Assume that for all $N\geqslant 1$ we can prove that there exists 
  an at most countable set of quasi-diagonals $\iota\colon C\hookrightarrow E^2$ such 
  that $\iota^*(\pr_1^*\L)^{\otimes N} = \iota^*(\pr_2^*\L)^{\otimes N}$. Then we 
  get the result since a countable union of countable sets is countable. Thus,
  further we fix $N\geqslant 1$, replace the line bundle $\L$ by $\L^{\otimes N}$
  and prove that there exists at most countable set of quasi-diagonals 
  $\iota\colon C\hookrightarrow E^2$ such that 
  \begin{equation}
  \label{_equation_condition_of_qd_}
  \iota^*(\pr_1^*\L) = \iota^*(\pr_2^*\L).
  \end{equation}
  {\bf Step 2:} Assume by contradiction that the set of 
  quasi-diagonals $\iota\colon C\hookrightarrow E^2$ satisfying 
  \eqref{_equation_condition_of_qd_} is uncountable. Since the 
  Hilbert scheme of curves in $E^2$ contains a countable set of components
  one can find a component there which contains an irreducible subvariety 
  $T\subset \mathrm{Hilb}(E^2)$ such that for each $t\in T$ the 
  corresponding curve $C_t\subset E^2$ satisfies \eqref{_equation_condition_of_qd_}.
  Denote by $X\subset E^2\times T$ the universal family over $T$. 
  Denote by $\alpha\colon X\to E^2$ the projection to the first component of the product.
  Denote by $\pi\colon X\to T$ the projection to the second component of the product.
  Note that the image $\alpha(X)$ cannot be 1-dimensional, because a 1-dimensional subvariety of $E^2$ contains only finitely many irreducible components, whereas our family of quasi-diagonals is uncountable. Since $E^2$ is irreducible, the map $\alpha$ must be surjective.
  Replace $T$ and $X$ by their resolutions of singularities if necessary.
  For each $t\in T$ denote by $X_t$ the preimage $X_t = \pi^{-1}(t)$.
  Finally, denote by $P_i\colon X\to E$ the composition $\pr_i\circ \alpha$.

  By assumption there exists an open subset $U\subset T$ such that
  for any $t\in U$ the image of the map restriction $\alpha_t = \alpha|_{X_t}\colon X_t \to E^2$
  is a quasi-diagonal such that $\alpha_t^*(\pr_1^*\L) = \alpha_t^*(\pr_2^*\L)$.
  This implies that the restriction of the line
  bundle $P_1^*L\otimes P_2^*L^{-1}$ to a general fiber $X_t$ of $\pi$ equals $\OP_{X_t}$.
  Thus, there exists a line bundle $\M$ on $T$ and 
  divisors $E_1,E_2\subset X\setminus \pi^{-1}(U)$ 
  such that one has the following equality:
  \begin{equation}\label{_equation_condition_of_qd_2_}
  P_1^*(\L) \otimes P_2^*(\L^{-1}) = \pi^*\M\otimes \OP_X(E_1-E_2).
  \end{equation}
  {\bf Step 3:} There exist a point $e\in E$ and an integer $M\geqslant
  1$ such that $\L = \OP_E(M e)$. Note that replacing $L$ by $L^{\otimes N}$ 
  and replacing $e$ by $e'$ such that the difference $e-e'$ is $N$-torsion in $E$ we can assume 
  that the preimage of $e$ under $P_1$ and $P_2$ is irreducible in $X$. 
  On the other hand, the equality  \eqref{_equation_condition_of_qd_2_} stays true.
  
  Consider the surjective map onto the curve 
  $P_i\colon X\to E$ and denote by $D_i\subset X$ the divisor $P_i^{-1}(e)$.
  Condition \eqref{_equation_condition_of_qd_2_} implies that there exist a 
  rational map $f\colon X\dashrightarrow \p^1$, divisors $\xi_i\subset X$ 
  and integers $k_i$ for $1\leqslant i\leqslant s$ 
  such that $\pi(\xi_i)$ is a proper subvariety in $T$ and 
  the following equality holds in the free abelian 
group $\Div(X)$ of divisors on $X$:
  \begin{equation}\label{_equation_more_precise_condition_of_qd_}
   M(D_1 - D_2)  - (f^{-1}(0) - f^{-1}(\infty)) = \sum_{i=1}^s k_i \xi_i,
  \end{equation}
  {\bf Step 4:} The divisor $ \sum_{i=1}^s k_i \xi_i$ is not linearly equivalent 
  to $0$ in $\Div(X)$ since otherwise so is $M(D_1 - D_2)$. Since 
  $\OP_X(M(D_1 - D_2) ) = P_1^*(\L) \otimes P_2^*(\L^{-1}) $ we get a contradiction 
  with \ref{lemma: F1 - F2 is not torsion} and \ref{lemma: pullback is torsion}. 

  In addition, the image $\pi(D_i)$ for $i = 1$ or $2$ cannot be a proper subvariety 
  of $T$. Otherwise, there exists a point $t\in T$ such that $P_i(\pi^{-1}(t)) = e$. 
  \ref{_lemma:_map_to_curve_factors_Lemma_}
implies that in this case $P_i$ factors through $\pi$
  i.e. $\alpha$ maps all quasi-diagonals in our family to fibers of $\pr_i\colon E^2\to E$
  which is impossible since all these fibers are not quasi-diagonals. 

\hfill

{\bf Step 5:} 
Passing to a power of the bundle $L$ and translating
by a torsion element of $E$, we can arrange
$D_i$ to be irreducible (Step 3). Since each 
quasi-diagonal projects surjectively to 
both factors of $E\times E$, the divisor $D_i$  
intersects each fiber of $\alpha:\; X\to E^2$.
Therefore, the divisors $D_1, D_2\subset X$
define a multisection of the projection $X\to T$.
This implies that the divisors $D_1$ and $D_2$
do not contain the components $\xi_i$. Therefore,
the equation 
\eqref{_equation_more_precise_condition_of_qd_}  
implies that the function $f$ is equal
to $\infty$ on each of $\xi_i$. 
Applying \ref{_lemma:_map_to_curve_factors_Lemma_}
to the map $f:\; X\to {\Bbb P}^1$, we obtain that
$f$ factors through $\pi:\; X\to T$.
Thus, $M(D_1 - D_2)$ is equal in $\Div(X)$ to the linear combination 
of preimages of proper subvarieties in $T$. 
However, this contradicts the previous 
observation that all components of $D_1$ and $D_2$ are
multisections of $\pi$. We get a contradiction; thus, the
assumption of existence of an uncountable family of
quasi-diagonals was wrong. This completes the proof.
 \endproof

\hfill

\remark
An independent proof of
\ref{_at_most_countably_many_quasi-diagonals_Theorem_}
is obtained by complex-analytic methods in 
Subsection \ref{_Countability_coreduction_Subsection_}.

%%%%%%%%%%%%%%%%%%%%%%%%%%%%%%%%%%%%%%%%%%%%%%%%%%%%%%%%%%%%
\subsection{Open questions about quasi-diagonals}
%%%%%%%%%%%%%%%%%%%%%%%%%%%%%%%%%%%%%%%%%%%%%%%%%%%%%%%%%%%%

Quasi-diagonals are a fascinating subject which is worth
further consideration. Before we close the subject, we
state a few interesting questions and observations.

\hfill

\remark By \ref{_corollary:_line_bundles_for_restricted_qd_Corollary} for all line bundles $L$, except countably many,
there are no restricted quasi-diagonals on $(E,L)$. We showed in Section \ref{section: lines don't give restricted qd}
that indeed many curves which could produce quasi-diagonals fail to do so. Moreover, all our constructions of restricted
quasi-diagonals work for line bundles with the special property.

\hfill

\question
Call a line bundle $L\in \Pic_n(E)$ {\bf commensurable with $\calo(n[e])$}
if $L\otimes \calo(n[e])^{-1}$ is a torsion.
Is it true that there are no restricted quasi-diagonals on
$(E,L)$ when $L$ is not commensurable with $\calo(n[e])$?

\hfill

\claim
For any even $n$ and any $L$ commensurable with $\calo(n[e])$,
there exist infinitely many restricted quasi-diagonals over $(E,L)$ in $E^n$.

\hfill

\proof The quasi-diagonal constructed in \ref{example: restricted quasi-diagonal for even}
is commensurable with $\calo(n[e])$. Then we conclude 
by \ref{_remark_translations_of_qd_by_torsion_}.
\endproof

\hfill

\question
Consider $E^n$, where $n$ is odd, and $E$ does not
admit complex multiplication. Are there any restricted
quasi-diagonals in $E^n$, for some ample line bundle $L$ on $E$?

\hfill

\question
Even in the case $n=3$ and $E = \C/\Z[\zeta]$ we have constructed 
only one example of a restricted quasi-diagonal in $E^3$ 
(\ref{_restricted_quasi-diagonal_for_3_Example_}).
Thus, $E^3$ contains a countable set of restricted quasi-diagonals 
as was shown in \ref{_remark_translations_of_qd_by_torsion_}.
However, one can ask whether $E^3$ contains any restricted quasi-diagonal
which is not a translation of the one constructed in 
\ref{_restricted_quasi-diagonal_for_3_Example_}?

\hfill

\question
Our only example (see \ref{example: non-elliptic quasi-diagonal}) 
of a quasi-diagonal which is not an elliptic 
curve is a singular curve in $E^2$. Moreover, 
all our examples of restricted quasi-diagonals
are elliptic curves. So it is interesting to ask: are there smooth 
non-elliptic quasi-diagonals? And also are there non-elliptic 
restricted quasi-diagonals?

%%%%%%%%%%%%%%%%%%%%%%%%%%%%%%%%%%%%%%%%%%%%%%%%%%%%%%%%%%%%%%%%%%%%%%%%

\section{Bogomolov-Guan manifolds}\label{_section_bg_manifolds}

%%%%%%%%%%%%%%%%%%%%%%%%%%%%%%%%%%%%%%%%%%%%%%%%%%%%%%%%%%%%%%%%%%%%%%%%

%%%%%%%%%%%%%%%%%%%%%%%%%%%%%%%%%%%%%%%%%%%%%%%%%%%%%%%%%%%%%%%%%%%%%%%%
\subsection{Principal toric and elliptic fibrations}
\label{_principal_toric_Subsection_}
%%%%%%%%%%%%%%%%%%%%%%%%%%%%%%%%%%%%%%%%%%%%%%%%%%%%%%%%%%%%%%%%%%%%%%%%

\definition
Let $M$ be a complex manifold equipped with a free holomorphic action
of an elliptic curve $E$. Then $M$ is called
{\bf a principal elliptic fibration}, over the base $M/E$.

\hfill

\remark \label{_remark_Fischer-Grauert_thm_}
Let $\pi:\; M \to X$ be a smooth holomorphic fibration
with fibers isomorphic to a given elliptic curve. Assume
that the monodromy of the corresponding Gauss-Manin local
system is trivial. Then the Fischer-Grauert  theorem
\cite{_Fischer_Grauert_}
implies that $M$ admits a structure of a principal
elliptic fibration over $X$.

\hfill

\definition
{\bf A principal toric bundle} is a complex manifold
equipped with a free holomorphic action
of a compact complex torus.
Just like for elliptic fibrations,
a principal toric bundle $M$ with an action
of $T$ is holomorphically fibered over $M/T$.

\hfill

Geometry and topology of principal toric bundles
were explored at great length in \cite{_Hofer:remarks_}.
Clearly, the set of isomorphism classes of 
principal $T$-bundles over a complex manifold $X$ 
is identified with $H^1(X, {\Bbb T})$, where ${\Bbb T}$ is the
sheaf whose group of sections ${\Bbb T}(U)$ over an
open set $U\subset X$ is the group of holomorphic
maps from $U$ to $T$, considered as a complex Lie group.
Suppose that $\dim_\C T=n$, and let $\Z^{2n}=\Lambda\stackrel \phi\rightarrow \C^n$
be the corresponding lattice. To compute $H^1(X, {\Bbb T})$,
we consider the following exact sequence of sheaves over $X$:
\[
0\to \Z_X^{2n}\stackrel \phi \to \calo_X^n \to {\Bbb T}\to 0.
\]
The corresponding long exact sequence gives
\begin{equation}\label{_exact_Hofer_Equation_}
0\to H^1(X, \Z_X^{2n})\stackrel \phi \to H^{0,1}(X)^n\to
H^1(X, {\Bbb T})\to H^2(X, \Z_X^{2n})\to \dots
\end{equation}
The $2n$ classes in $H^2(X, \Z_X^{2n})$ associated with
a given principal $T$-bundle are called {\bf the Chern
classes} of this $T$-bundle.

\hfill

\example \label{_example_quotient_of_totalizalion_of_line_bundle_}
Let $L$ be a holomorphic line bundle over a complex manifold $X$,
and $\Tot^\circ(L)$ the space of non-zero vectors in its total space,
considered as a $\C^*$-fibration over $X$.
Fix a complex number $\lambda\in \C$, $|\lambda|>1$.
The group $\C^*$ acts on $\Tot^\circ(L)$, transitively on
fibers.  Consider the quotient $\Tot^\circ(L)/\langle\lambda\rangle$;
it is a principal elliptic bundle on $X$, with fibers
elliptic curves isomorphic to $\C^*/\langle\lambda\rangle$.
It is easy to see that the corresponding Chern classes are $(0, c_1(L))$.

\hfill

%%%%%%%%%%%%%%%%%%%%%%%%%%%%%%%%%%%%%%%%%%%%%%%%%%%%%%%%%%%%
\theorem\label{_Blanchard_Kahler_Theorem_}
(Blanchard)\\
Let $\pi:\; M \to X$ be a holomorphic principal toric fibration,
with $X$ K\"ahler. Then $M$ admits a K\"ahler structure
if and only if the Chern classes of $\pi$ vanish in $H^2(X, \Q)$
(equivalently, the Chern classes of $\pi$ are torsion).

\proof
\cite{_Blanchard:complexes_}. \endproof

\hfill

In order to describe the properties of varieties which are going to arise
further, we need the following definition.

\hfill

\definition
A complex variety is called {\bf Fujiki class C}
if it is bimeromorphic to a K\"ahler manifold.
A complex variety is called {\bf Moishezon}
if it is bimeromorphic to a projective manifold.

\hfill

%%%%%%%%%%%%%%%%%%%%%%%%%%%%%%%%%%%%%%%%%%%%%%%%%%%%%%%%%%%%
\theorem\label{_projective_toric_Theorem}
Let $\pi:\; M \to X$, $X=M/T$ be a holomorphic principal toric fibration,
with $X$ and $T$ projective. Assume that the Chern classes
of $\pi$ are torsion. Then $M$ is projective if and 
only if the corresponding class in $H^1(X, {\Bbb T})$
is torsion. 

\hfill

\proof If $M$ is a projective variety then there exists a
multisection of $\pi$, that is,
 an irreducible projective subvariety $Z\subset M$
 such that $\pi(Z) = X$ and the map $\pi|_{Z}\colon Z\to X$ 
 is generically finite.

 Consider the fiber product $M_Z = M\times_X Z$ with the induced 
 principal toric fibration $\pi_Z\colon M_Z\to Z$. This fibration
 admits a global section; thus, it is isomorphic to 
 a direct product $M_Z\cong Z\times T$ where $T$ is the fiber of 
 $\pi$.
 
 The class of the fibration $\pi_Z$ vanishes and equals the image of 
 the class of the fibration $\pi$ under the inverse image map $H^1(X,{\Bbb T})\to H^1(Z,{\Bbb T})$. 
 By \eqref{_exact_Hofer_Equation_}, this
is possible only if the class of the fibration $\pi$ is torsion.

 Note that this proof works for a complex manifold $M$ equipped
with a principal toric fibration: once
 we have a multisection we observe that the inverse image of the fibration is a product.

 In order to prove another implication, we need the following construction.
 Recall that if $\zeta$ is a cocycle in $H^1(X,{\Bbb T})$ and $M_{\zeta}$ is a principal 
 toric bundle associated to $\zeta$ over $X$ then for any positive integer $k$
 there exists the following surjective unramified map:
 \[
 \alpha\colon M_{\zeta}\to M_{k \zeta}.
 \]
 One can define $\alpha$ on local trivializations $U\times T$; we set $\alpha((u,t)) = (u,k\cdot t)$
 and one can easily check that this map is compatible with the cocycle. 

 Now if $\pi\colon M\to X$ is a principal  toric fibration whose $c_1(\pi)$ is torsion,
 and corresponding class $\zeta\in H^1(X,{\Bbb T})$ is torsion then by the above construction
 there exists a surjective unramified map $\alpha\colon M\to X\times T$. By our assumption 
 the product $X\times T$ is projective. Therefore, so is $M$ since the inverse image of an ample 
 line bundle under $\alpha$ is ample on $M$.

 Finally, if $M$ is a Moishezon variety then one can construct a subvariety $Y\subset M$ such that 
 the map $\pi|_{Y}\colon Y\to X$ is generically finite. Thus, $M\times_X Y$ is isomorphic 
 to $Y\times T$. Therefore, the cocycle in $H^1(X,{\Bbb T})$ which defines the fibration $M$ is torsion, following the same logic as in the first part of the proof. 
 We have shown that $M$ is Moishezon if and only if $M$ is projective.
 \endproof

%%%%%%%%%%%%%%%%%%%%%%%%%%%%%%%%%%%%%%%%%%%%%%%%%%%%%%%%%%%%%%%%%%%%%%%%
\subsection{Principal toric fibrations associated with
  holomorphic line bundles}
%%%%%%%%%%%%%%%%%%%%%%%%%%%%%%%%%%%%%%%%%%%%%%%%%%%%%%%%%%%%%%%%%%%%%%%%

The constructions which are given in this subsection
are central in locally conformally K\"ahler geometry,
\cite{_OV:Principles_}; we already used them when we
defined the Kodaira surface. 

\hfill

\definition 
Let $L_1, ..., L_n$ be a collection of line bundles on a
complex manifold $M$,
and $\lambda_1, ..., \lambda_n$ complex numbers, $|\lambda_i|>1$.
Consider the fibered product $\prod_{i=1}^n
\frac{\Tot^\circ(L_i)}{\langle\lambda_i\rangle}$;
it is a principal toric fibration with a fiber
$\prod_i^n\frac{\C^*}{\langle\lambda_i\rangle}$
When $\lambda_1= \lambda_2=\dots =\lambda_n=\lambda$, we denote it
as $R_\lambda(M,L_1, ..., L_n)$.

\hfill

\proposition\label{_prop_pullback_of_line_bundle_}
Consider the projection
$P:\; R_\lambda(M,L_1, ..., L_n)\arrow M$.
Then $P^*(c_1(L_i))=0$.

\hfill

\proof To prove that 
$P^*(c_1(L_i))=0$, it would suffice to show that
$P^*L_i$ has a nowhere vanishing section on $\prod_M S_i$.
This is clear because 
 for each $(t_1, ..., t_n)\in \prod_M S_i$, the point
$t_i$ defines a unit vector in $P^* L_i\restrict{(t_1, ..., t_n)}$.
\endproof

\hfill

\lemma
\label{_Kodaira-type_manifolds_Lemma_} 
Let $L_1$ and $L_2$ be two line bundles on a smooth complex manifold $M$. If $F$ is the elliptic curve $F = \C^*/\langle\lambda\rangle$ with the diagonal action on $R_{\lambda}(M,L_1,L_2)$, then one has the following isomorphism:
\[
R_{\lambda}(M,L_1,L_2)/F\cong R_{\lambda}(M, L_1\otimes L_2^{-1}).
\]

\hfill

\proof
The lemma is proved in \cite[Lemma 6.4]{_BKKY_}, we repeat the proof here since 
it is easy and important. 

We further denote by $\mathcal{L}$ the line bundle $\mathcal{L} = L_1\otimes L_2^{-1}$. Let $\Phi$ be the following map:
\begin{align*}
\Phi\colon \Tot^\circ (L_1)\times_M \Tot^\circ (L_2)\to \Tot^\circ (\mathcal{L}); && \Phi(p, s_1,s_2) = (p, s_1\otimes s_2^{\vee}).
\end{align*}
Here $p$ is a point in $M$ and $s_i$ is an element in $\Tot^\circ (L_i)$ and $s_2^{\vee}$ is the unique linear operator on $\Tot(L_2)$ such that $s_2^{\vee}(s_2) = 1$. One can check that $\Phi$ extends to a surjective map $R_{\lambda}(M,L_1,L_2)\to R_{\lambda}(M,\mathcal{L})$. Moreover, the map $\Phi$ factors through the diagonal action of $F$ and the induced map $R_{\lambda}(M,L_1,L_2)/F\to R_{\lambda}(M,\mathcal{L})$ is an isomorphism.
\endproof

\hfill

\corollary\label{_Kodaira-type_manifolds_corollary_}
Let $L_1,L_2,\dots, L_n$ be $n$ line bundles on a smooth complex manifold $M$. If $F$ is the elliptic curve $F = \C^*/\langle\lambda\rangle$ with the diagonal action on $R_{\lambda}(M,L_1,L_2,\dots,L_n)$, then one has the following isomorphism:
\[
R_{\lambda}(M,L_1,\dots, L_n)/F\cong R_{\lambda}(M, L_2\otimes L_1^{-1})\times_M\dots\times_M R_{\lambda}(M, L_n\otimes L_1^{-1}).
\]

\hfill

\proof
See \cite[Lemma 6.6]{_BKKY_}.
\endproof

%%%%%%%%%%%%%%%%%%%%%%%%%%%%%%%%%%%%%%%%%%%%%%%%%%%%%%%%%%%%%%%%%%%%%%%%

%%%%%%%%%%%%%%%%%%%%%%%%%%%%%%%%%%%%%%%%%%%%%%%%%%%%%%%%%%%%%%%%%%%%%%%%
\subsection{Hilbert schemes}
%%%%%%%%%%%%%%%%%%%%%%%%%%%%%%%%%%%%%%%%%%%%%%%%%%%%%%%%%%%%%%%%%%%%%%%%

For definition, basic properties and other results on
Hilbert scheme of points on a complex surface, see \cite{_Nakajima:Hilbert_}.

\hfill

\definition
A {\bf Hilbert scheme of points} on a variety $M$
is the Hilbert space of ideal sheaves $I\subset \calo_M$
with $F:=\calo_M/I$ supported in a finite subset of $M$; dimension of $H^0(F)$
is called {\bf length}. The Hilbert scheme of points
for length $n$ is denoted by $M^{[n]}$.

\hfill

\remark When $M$ is a complex surface, $M^{[n]}$
is a smooth resolution of the $n$-th symmetric power
of $M$ (\cite{_Nakajima:Hilbert_}).

\hfill

\remark If the surface $M$ is holomorphically symplectic,
$M^{[n]}$ is also holomorphically symplectic (this
observation, made in \cite{_Beauville_}, easily follows 
from Serre duality).

\hfill

\remark
The Hilbert scheme of a K3 surface is a simply connected, 
holomorphically symplectic manifold. The Hilbert scheme
of a torus $T$ is not simply connected, but the fiber
of its Albanese map $T^{[n]}\arrow T$ 
has finite fundamental group. The universal cover
of this fiber is called {\bf a generalized Kummer variety}.

\hfill

This way one obtains the two main examples of simply connected
holomorphically symplectic projective manifolds. The third
series of examples, associated with the Hilbert scheme of Kodaira
surface, is non-projective and non-K\"ahler.

%%%%%%%%%%%%%%%%%%%%%%%%%%%%%%%%%%%%%%%%%%%%%%%%%%%%%%%%%%%%%%%%%%%%%%%%
\subsection{Kodaira surface}
%%%%%%%%%%%%%%%%%%%%%%%%%%%%%%%%%%%%%%%%%%%%%%%%%%%%%%%%%%%%%%%%%%%%%%%%

\definition\label{_definition_Kodaira_}
Let $L$ be a line bundle
on an elliptic curve $E$ with the first Chern class
$c_1(L)\neq 0$. Denote by $\tilde S$ the corresponding $\C^*$-bundle
on $E$ obtained by removing the zero section,
$\tilde S=\Tot(L) \backslash 0$. Fix a complex
number $\lambda$ with $|\lambda| > 1$, and
let $h_\lambda:\; \tilde S \arrow \tilde S$
be the corresponding homothety of $\tilde S$.
The quotient $\tilde S/\langle h_\lambda\rangle$
is called {\bf a primary Kodaira surface},
or {\bf Kodaira-Thurston surface}, or
{\bf Kodaira surface}.

\hfill

\remark \label{_remark_Kodaira_is_elliptic_fibration_}
The Kodaira surface is a principal elliptic fibration over 
the elliptic curve $E$, with the fiber
identified with the elliptic curve $E_L:\;= \C^*/\langle \lambda\rangle$.
Therefore, it is holomorphically symplectic.

\hfill

\remark
{\bf The secondary Kodaira surfaces} are obtained from
the primary Kodaira surfaces as quotients by a free
holomorphic action of a finite group. These surfaces
are elliptically fibered over $\C P^1$; they are not
holomorphically symplectic.

\hfill

\remark \label{remark_DB-morphism_is_projective_}
Though $\Kod$ is not a projective surface, one still can define its 
Hilbert scheme of points $\Kod^{[n]}$. Denote by $\Kod^{(n)}$
the $n$-th symmetric power of $\Kod$ i.e. the quotient of the product 
$\Kod^n$ by the natural action of the symmetric group $\SG_n$. 
Note that the symmetric power $\Kod^{(n)}$ is a singular variety.
By  \cite[Theorem 2.3.1]{_Cataldo_Migliorini:Douady_space_} there exists 
the Douady--Barlet morphism $\Kod^{[n]}\to \Kod^{(n)}$ which resolves 
the singularity of $\Kod^{(n)}$. Moreover, this morphism
is projective \cite[Theorem 2.3.1]{_Cataldo_Migliorini:Douady_space_}.

By \ref{_remark_Kodaira_is_elliptic_fibration_} the Kodaira surface 
is a principal elliptic fibration with a fiber $E_L$.
Since the action of $E_L$ on $\Kod^{[n]}$ and $\Kod^{(n)}$ is equivariant 
with respect to the Douady--Barlet morphism, the induced morphism 
$\Kod^{[n]}/E_L\to \Kod^{(n)}/E_L$ is also projective.

%%%%%%%%%%%%%%%%%%%%%%%%%%%%%%%%%%%%%%%%%%%%%%%%
\subsection{ Nilmanifolds }
\label{_Nilmanifold_Subsection_}
%%%%%%%%%%%%%%%%%%%%%%%%%%%%%%%%%%%%%%%%%%%%%%%%

For some proofs given at the end of this paper, we need to 
do some explicit computations in the de Rham algebra of a power of 
Kodaira surface. This is most easy to do if we interpret the Kodaira
surface as a nilmanifold
(\cite{_Hasegawa_,_OV:Principles_}). 
Here we introduce this notion
and some basic results, mostly due to A. Maltsev (\cite{_Maltsev_});
for proofs and more details, see \cite{_Corwin_Greenleaf_}.

\hfill

\definition
Let $M$ be a compact manifold equipped
with a transitive action of a nilpotent Lie group.
Then $M$ is called {\bf a nilmanifold}.

\hfill

\remark
As shown by A. Maltsev, all nilmanifolds are obtained as quotient spaces,
$M=G/\Gamma$, where $\Gamma\subset G$ is a discrete,
cocompact subgroup, acting by translations.

\hfill

\theorem  (Maltsev)\\ \nopagebreak
Let $\g$ be a nilpotent Lie algebra defined over $\Q$,
and $G$ its Lie group. Then $G$ contains a discrete
subgroup $\Gamma$ such that $G/\Gamma$ is compact, and
$\Gamma= e^{\Gamma_\g}$, where $\Gamma_\g=\Z^n$ is a lattice
 in $\g$.
Finally, all nilmanifolds are obtained this way.

\proof \cite{_Maltsev_}. \endproof

\hfill

\remark Topologically, all simply connected
nilpotent Lie groups are diffeomorphic to $\R^n$, and all nilmanifolds are 
iterated circle fibrations.

%%%%%%%%%%%%%%%%%%%%%%%%%%%%%%%%%%%%%%%%%%%%%%%%%%%%%%%%%%%%%%%%%%%%%%%%
\subsection{Kodaira surface as a nilmanifold}
\label{_Kod_nilma_Subsection_}
%%%%%%%%%%%%%%%%%%%%%%%%%%%%%%%%%%%%%%%%%%%%%%%%%%%%%%%%%%%%%%%%%%%%%%%%

%%%%%%%%%%%%%%%%%%%%%%%%%%%%%%%%%%%%%%%%%%%%%%%%%%%%%%%%%%%%
\definition\label{_integra_comple_on_Lie_Definition_}
{\bf An integrable complex structure} on a real Lie algebra
$\g$ is a subalgebra $\g^{1,0}\subset \g \otimes_\R \C$
such that $\g^{1,0}\oplus \overline{\g^{1,0}} = \g \otimes_\R \C$.

\hfill

\remark
Any such decomposition defines a complex structure $I$
on $\g$ by $I\restrict{\g^{1,0}}=\1$ and
$I\restrict{\g^{0,1}}=-\1$.
We extend this operator to a right-invariant almost
complex structure operator ${\cal I}$ on $G$.
Integrability of ${\cal I}$ is equivalent to $[T^{1,0}G, T^{1,0}G]\subset T^{1,0}G$,
which is equivalent to $[\g^{1,0}, \g^{1,0}]\subset \g^{1,0}$.

\hfill

\remark
Right-invariant complex structures on a connected real Lie group
are in 1 to 1 correspondence with integrable
complex structures on its Lie algebra.

\hfill

\definition
Let $\Gamma\subset G$ be a cocompact lattice in a nilpotent Lie group.
{\bf A complex nilmanifold} is a quotient
$M:=G/\Gamma$ equipped with a complex structure, in such a way
that $G$ has a right-invariant complex structure,
and the projection $G \arrow M$ is holomorphic.

\hfill

Let $G:= \R \times G_0$, where $G_0$ is the 3-dimensional 
real Lie group of 3x3 upper triangular matrices.
This group contains many cocompact lattices, for example $\Gamma :=\Z \times \Gamma_0$,
where $\Gamma_0$ is the group of strictly upper triangular 
integral matrices.  The corresponding
Lie algebra $\g$ is generated by $x,y, z, t$
with the only non-zero commutator $[x, y]=z$.

\hfill

\claim
The Kodaira surface can be equivalently 
defined as $M:=G/\Gamma$ with the
complex structure defined by the subalgebra
$\g^{1,0}:= \langle x+ \1 y, z+\1 t\rangle\subset
\g\otimes_\R \C$, 
where $G= \R \times G_0$ is the group
defined above, $\g=\Lie(G)$, and $\Gamma\subset G$ is a cocompact lattice.

\proof
\cite{_Hasegawa_}
\endproof

%%%%%%%%%%%%%%%%%%%%%%%%%%%%%%%%%%%%%%%%%%%%%%%%%%%%%%%%%%%%
\subsection{Bogomolov-Guan manifolds }\label{_section_BG_construction_}
%%%%%%%%%%%%%%%%%%%%%%%%%%%%%%%%%%%%%%%%%%%%%%%%%%%%%%%%%%%%

\definition
Let $\Kod$ be a Kodaira-Thurston surface,
 $\Kod^{[n]}$ its Hilbert scheme, $\Kod^{(n)}$ its symmetric power. 
 Let $\pi:\; \Kod \arrow E$
be the elliptic fibration constructed above, $E_1$ its fiber,
and $\pi^{(n)}\colon \Kod^{(n)}\to E^{(n)}$ the 
induced map between the symmetric powers.
Denote by $r:\; \Kod^{[n]} \arrow \Kod^{(n)}$ 
the resolution map.
Recall that $\sigma\colon E^{(n)}\to E$ denotes
the map that sums up $n$ elements in~$E$.
Then the following composition is an isotrivial fibration:
\begin{equation*}
p = \sigma\circ\pi^{(n)}\circ r\colon \Kod^{[n]}\to E.
\end{equation*}
Denote a fiber of $p$ over the zero point $e\in E$  by $F$.
Then $F$ is a smooth divisor in 
a holomorphically symplectic manifold $(\Kod^{[n]}, \Omega)$.
The restriction of $\Omega$ to $F$ has rank 
$2n-2$, because $F\subset \Kod^{[n]}$ is a divisor.
Denote by $K\subset TF$ the kernel of $\Omega\restrict {F}$,
that is, the set of all $x\in TF$ such that 
$\Omega\restrict F(x, \cdot)=0$.

\hfill

\remark
The corresponding foliation is called {\bf
the characteristic foliation}; this definition is valid
for any divisor $D$ in a holomorphically symplectic manifold,
and the characteristic foliations have interesting
algebraic geometry indicative of the Kodaira dimension
of the line bundle $\calo(D)$, see
\cite{_Amerik_Campana:Characteristic_,_Amerik_Guseva_,_Abugaliev_,_Anella_Huybrechts:Characteristic_}.

\hfill

\remark Note that the diagonal action of $E_1$ on $\Kod^{(n)}$ 
induces the action on $\Kod^{[n]}$ and restricts to the action on $F$.
The form $\Omega$ is invariant under the action of $E_1$. Thus, 
the leaves of the characteristic foliation $K$ are in fact the orbits 
of the action of $E_1$.

\hfill

\remark \label{remark: definition of BG}
The leaf space $W$ of $K$ 
is a holomorphically symplectic 
orbifold, it is isomorphic to the quotient manifold 
$F/E_1$, but it is never smooth.
When the degree of the line bundle $L$ over $E$ is
divisible by $n$,
the space $W$ has a smooth finite covering, of order
$n^2$, ramified in the singular points of $W$.
This covering is called 
{\bf the Bogomolov-Guan manifold}, we denote it $BG$. 
By construction, it is compact, simply connected,
holomorphically symplectic. The construction of $BG$
can be described with the following diagram:
% https://q.uiver.app/#q=WzAsOCxbMywwLCJcXG1hdGhybXtLb2R9Xntbbl19Il0sWzMsMSwiRV57KG4pfSJdLFszLDIsIkUiXSxbMiwyLCJ4Il0sWzIsMSwiXFxtYXRoYmJ7UH1ee24tMX0iXSxbMiwwLCJGIl0sWzEsMCwiVyJdLFswLDAsIkJHIl0sWzAsMSwiXFxwaV57KG4pfVxcY2lyYyByIl0sWzEsMiwiXFxTaWdtYSJdLFszLDIsIiIsMix7InN0eWxlIjp7InRhaWwiOnsibmFtZSI6Imhvb2siLCJzaWRlIjoidG9wIn19fV0sWzQsM10sWzQsMSwiIiwyLHsic3R5bGUiOnsidGFpbCI6eyJuYW1lIjoiaG9vayIsInNpZGUiOiJ0b3AifX19XSxbNSw0XSxbNSwwLCIiLDAseyJzdHlsZSI6eyJ0YWlsIjp7Im5hbWUiOiJob29rIiwic2lkZSI6InRvcCJ9fX1dLFs1LDYsIkVfMSIsMl0sWzcsNiwicSJdLFs2LDRdLFs3LDQsIlAiLDJdXQ==
\[\begin{tikzcd}
	BG & W & F & {\mathrm{Kod}^{[n]}} \\
	&& {\mathbb{P}^{n-1}} & {E^{(n)}} \\
	&& e & E
	\arrow[from=1-1, to=1-2]
	\arrow["P"', from=1-1, to=2-3]
	\arrow[from=1-2, to=2-3]
	\arrow["{E_1}"', from=1-3, to=1-2]
	\arrow[hook, from=1-3, to=1-4]
	\arrow[from=1-3, to=2-3]
	\arrow["{\pi^{(n)}\circ r}", from=1-4, to=2-4]
	\arrow[hook, from=2-3, to=2-4]
	\arrow[from=2-3, to=3-3]
	\arrow["\sigma", from=2-4, to=3-4]
	\arrow[hook, from=3-3, to=3-4]
\end{tikzcd}\]

\hfill

\remark
\label{_remark_base_of_BG}
The action of $E_1$ described above commutes with the map $(\pi^{(n)}\circ r)|_{F}\colon F\to E^{(n)}$; thus, this map factors through the leaf space $W$.
Since $BG$ is a finite cover of $W$ we get a map $P\colon BG\to E^{(n)}$. 
The image of this map is the quotient variety $X/{\SG_{n}}$, where
\[
X= \left\{(x_1, ..., x_{n})\in E^{n}\ \  \left|\ \  \sum_i x_i =e\right.\right\} \cong E^{n-1},
\]
and $\SG_{n}$ is the symmetric group naturally acting on $X\subset E^{n}$. 
One can check that $X/{\SG_{n}}\cong\p^{n-1}$; therefore, the image of $P$ is the divisor $\p^{n-1}\subset E^{(n)}$.
By \cite[Theorem A]{_BKKY_} the map $P$ is an algebraic reduction of $BG$.

\hfill

\remark Since the  Bogomolov-Guan manifold
contains a blown-up Kodaira surface,
it is non-K\"ahler.

\hfill

\remark \label{_remark_subvarieties_in_BG_and_Kod_mod_E1_}
 The diagonal action of $E_1$ on $\Kod^n$ commutes with the projection $\pi^n\colon \Kod^n\to E^n$. 
 Consider the induced map $\pi^n/E_1\colon \Kod^n/E_1\to E^n$ and denote by $M$ the preimage 
 $M = (\pi^n/E_1)^{-1}(X)$ where $X\subset E^n$ is defined in \ref{_remark_base_of_BG}.
 Then the construction induces a generically finite meromorphic map $\tau\colon M\to W$. 
  More precisely, the map 
 $\tau$ factors through  the quotient by the symmetric group $\SG_n$ and then the resolution of its singularities:
% https://q.uiver.app/#q=WzAsOCxbMSwwLCJXIl0sWzMsMCwiTSJdLFs0LDAsIlxcS29kXm4vRV8xIl0sWzIsMSwiTS9cXFN5bV9uIl0sWzAsMCwiQkciXSxbMywyLCJYIl0sWzQsMiwiRV5uIl0sWzIsMiwiXFxwXntuLTF9Il0sWzEsMiwiIiwwLHsic3R5bGUiOnsidGFpbCI6eyJuYW1lIjoiaG9vayIsInNpZGUiOiJ0b3AifX19XSxbMSwzXSxbMCwzXSxbMSwwLCJcXHRhdSIsMix7InN0eWxlIjp7ImJvZHkiOnsibmFtZSI6ImRhc2hlZCJ9fX1dLFs0LDBdLFs1LDYsIiIsMCx7InN0eWxlIjp7InRhaWwiOnsibmFtZSI6Imhvb2siLCJzaWRlIjoidG9wIn19fV0sWzEsNSwiKFxccGlebil8X00iXSxbMiw2LCJcXHBpXm4iXSxbNSw3XSxbMyw3XV0=
\[\begin{tikzcd}
	BG & W && M & {\Kod^n/E_1} \\
	&& {M/\SG_n} \\
	&& {\p^{n-1}} & X & {E^n}
	\arrow[from=1-1, to=1-2]
	\arrow["Q", from=1-2, to=2-3]
	\arrow["\tau"', dashed, from=1-4, to=1-2]
	\arrow[hook, from=1-4, to=1-5]
	\arrow[from=1-4, to=2-3]
	\arrow["{(\pi^n/E_1)|_M}", from=1-4, to=3-4]
	\arrow["{\pi^n/E_1}", from=1-5, to=3-5]
	\arrow[from=2-3, to=3-3]
	\arrow[from=3-4, to=3-3]
	\arrow[hook, from=3-4, to=3-5]
\end{tikzcd}\]
Note that by \ref{remark_DB-morphism_is_projective_} the morphism $Q\colon W\to M/\SG_n$ is projective.
 Consider now a subvariety $A\subset BG$, denote its image in $W$ by $B$ and denote by $\widetilde{A}$ the proper 
 preimage of $B$ in $M$. Since the map $BG\to W$ is finite then $A$ is Moishezon (resp. Fujiki class C) if and only if 
 so is $B$ and its image in $M/\SG_n$. Since $\tau$ is projective $B$ is  Moishezon (resp. Fujiki class C) if and only if 
 so is $\tau(B)$. Finally, since the map $M\to M/\SG_n$ is finite $\tau(B)$ is  Moishezon (resp. Fujiki class C) if and only if 
 so is $\widetilde{A}$.

%%%%%%%%%%%%%%%%%%%%%%%%%%%%%%%%%%%%%%%%%%%%%%%%%%%%%%%%%%%%%%%%%%%%%%%%
\subsection{Subvarieties of Bogomolov-Guan manifolds }
%%%%%%%%%%%%%%%%%%%%%%%%%%%%%%%%%%%%%%%%%%%%%%%%%%%%%%%%%%%%%%%%%%%%%%%%

\remark
Since the projection of the Kodaira surface to $E$
is Lagrangian, the map $P\colon  BG\arrow \p^{n-1}$ is a Lagrangian fibration.

\hfill

\remark
Let $E_1, E$ be elliptic curves,
and  $\Kod$ a Kodaira surface which is a principal $E_1$-bundle over $E$.
Therefore, a general fiber of $P$ is $E_1^{n-1}$, and every fiber of $P$ is a projective 
variety of the form $E_1^{n-1}/G_x$ for a finite group $G_x\subset\Aut(E_1^{n-1})$; in other 
words, all fibers of $P$
are finite quotients of $E_1^{n-1}$ (not resolutions of such quotients). 

\hfill

The following results were obtained (in a different form)
in \cite{_BKKY_}.

\hfill

%%%%%%%%%%%%%%%%%%%%%%%%%%%%%%%%%%%%%%%%%%%%%%%%%%%%%%%%%%%%
\theorem \label{_BKKY_main_Theorem_}
Let $A\subset BG$ be an irreducible complex subvariety,
and let  $P:\;  BG\arrow \p^{n-1}$ be the Lagrangian
fibration. Assume that $P(A)$ is positive-dimensional. Consider
the $\SG_{n}$-quotient map $\sigma:\; X\cong E^{n-1} \arrow \p^{n-1}$. 
Then the following assertions hold:
\begin{enumerate}
    \item[\textup{(1)}]  the manifold $A$ is Moishezon 
if and only if each irreducible component $Z$
of $\sigma^{-1}(P(A))$ is a restricted quasi-diagonal;
    \item[\textup{(2)}] the manifold $A$
is Fujiki class C if each irreducible component $Z$
of $\sigma^{-1}(P(A))$ is a curve in $E^{n-1}$
such that all projections $p_i:\; Z\arrow E$ have the same
degree.
    \item[\textup{(3)}] Whenever a subvariety of
 a Bogomolov-Guan manifold is Moishezon, it is
also projective.
\end{enumerate}
    
\proof 
Denote by $\widetilde{A}$ the proper preimage of $A$ in the quotient variety
$\Kod^n/E_1$. By \ref{_remark_subvarieties_in_BG_and_Kod_mod_E1_} the variety 
$A$ is Moishezon (resp. Fujiki class C) if and only if so is $\widetilde{A}$.

If $\dim(\sigma^{-1}(P(A)))\geqslant 2$, we can assume that $p_1\times p_2\colon Z\to E^2$ is surjective.
By \cite[Lemma 6.12]{_BKKY_} there exists a surjective map $\widetilde{A}\to \Kod^2/E_1$.
Since $\Kod^2/E_1$ is not Fujiki class C, then so are $\widetilde{A}$ and $A$.

Assume that $Z$ is an irreducible curve in $ \sigma^{-1}(P(A))\subset E^n$. 
Then its preimage $X$ in $\Kod^n/E_1$ is isomorphic to the fiber
product of $n-1$ principal elliptic bundles over $Z$ by \cite[Lemma 6.6]{_BKKY_}:
\[
X = X_1\times_{Z} X_2\times_Z\dots\times_Z X_{n-1}.
\]
Each bundle $\pi_i\colon X_i\to Z$ is equal to the quotient of the total space of 
the line bundle $p_i^*L\otimes p_1^*L^{-1}$ by  $\langle\lambda\rangle$.

If all projections $p_i:\; Z\arrow E$ 
do not have the same degree then one can find a surface $X_i$ 
which is not Fujiki class C by \cite[Lemma 6.3]{_BKKY_}.
Moreover, all subvarieties of $X_i$ are fibers over $Z$; thus, $A\subset X$ 
projects surjectively to $X_i$ so it is not Fujiki class C. 

If each $Z$ is a curve in $E^{n-1}$
such that all projections $p_i:\; Z\arrow E$ have the same
degree then in view of \ref{_example_quotient_of_totalizalion_of_line_bundle_} and \ref{_Blanchard_Kahler_Theorem_}
the preimage of each $Z$ is Fujiki class C; thus $A$ is also.

If $Z$ is not a quasi-diagonal then there exists $i$ such that 
$X_i$ is not Moishezon by \ref{_projective_toric_Theorem}. 
Thus, the algebraic dimension of $X_i$ equals $1$ and it contains no 
multisections. Therefore, $A$ projects surjectively to $X_i$ so $A$ is 
not Moishezon itself by \cite[Lemma 2.9]{_BKKY_}.

Finally, if $Z$ is a quasi-diagonal then $X$ is projective by
\ref{_example_quotient_of_totalizalion_of_line_bundle_} and
\ref{_projective_toric_Theorem}. Therefore, $A$ is also projective. 

Note also that if $Z$ is a quasi-diagonal and lies in 
$\sigma^{-1}(P(A))$ then necessarily it is a restricted quasi-diagonal.
This finishes the proof.
\endproof

%%%%%%%%%%%%%%%%%%%%%%%%%%%%%%%%%%%%%%%%%%%%%%%%%%%%%%%%%%%%%%%%%%%%%%%%

\section{Bogomolov-Guan precursors}\label{_section_bg_precursors}

%%%%%%%%%%%%%%%%%%%%%%%%%%%%%%%%%%%%%%%%%%%%%%%%%%%%%%%%%%%%%%%%%%%%%%%%

%%%%%%%%%%%%%%%%%%%%%%%%%%%%%%%%%%%%%%%%%%%%%%%%%%%%%%%%%%%%
\subsection{Bogomolov-Guan precursor}
%%%%%%%%%%%%%%%%%%%%%%%%%%%%%%%%%%%%%%%%%%%%%%%%%%%%%%%%%%%%

\definition
Let $E_1, E$ be elliptic curves, and 
$\pi\colon \Kod\to  E$ a Kodaira surface.
Denote by $D$ the fiber of the composition
$\left(\sigma\circ\pi^n\right)\colon \Kod^n\to E$ over the zero point $e$.
Consider the diagonal $E_1$-action on $D$.
{\bf The restricted Bogomolov-Guan precursor}
is the quotient $X = D/E_1$, and
{\bf the unrestricted Bogomolov-Guan precursor}
is $\Kod^n/E_1$.

\hfill

\hfill

\proposition
 The orbits of the $E_1$-action on $D$ are leaves of
  the characteristic foliation $\ker \Omega$ on $D$.

\hfill

\proof
Let $\sigma:\; E^{n} \arrow E$ be the function
$(z_1,..., z_{n})\to \sum_{i=1}^{n} z_i$. Then 
$\ker d\sigma= TD$, where $d\sigma$ is a 1-form obtained as a
derivative of $\sigma$, which is locally defined as a holomorphic 
map to $\C$. Locally in $\Kod$, we choose
a coordinate system $v, w$ such that $v$ is the
coordinate on the base of $p$ and $w$ the coordinate on the fiber.
By construction, the symplectic form $\Omega$
on $\Kod$ pairs the coordinate vector field on the fiber of $p$ with the
coordinate vector field on the base: locally on $\Kod$,
its holomorphically symplectic form is $dv\wedge dw$. Therefore, $d\sigma$ is
proportional to $\Omega(W, \cdot)$, where $W= \sum w_i$.
This implies that the characteristic
  foliation of $D$ is tangent
to $W$, and hence tangent to the diagonal $E_1$-action.
\endproof

\hfill

\remark The space of leaves of the characteristic foliation
is clearly holomorphically symplectic; this implies, in particular, the following fact.

\hfill

\corollary The restricted Bogomolov-Guan precursor $D/E_1$
is holomorphically symplectic,
and the natural map $P:\; D/E_1 \arrow X\subset E^n$ 
to the divisor $X=\{(z_1, ...,z_{n})\in E^{n}\ \ |\ \ \sum z_i =e\}$
is a Lagrangian fibration.

\hfill

\claim
The Bogomolov-Guan manifold admits a 
holomorphic, bimeromorphic map 
to a finite quotient of the restricted Bogomolov-Guan precursor.

\hfill

\proof
Indeed, $F/E_1$ is a finite quotient of the Bogomolov-Guan manifold,
and it is bimeromorphic to $\frac{D/E_1}{\SG_{n}}$.
\endproof

\hfill

\claim The Bogomolov-Guan precursor (restricted and
unrestricted) is a complex
nilmanifold.

\hfill

\proof This observation follows immediately from the construction
given in Subsection \ref{_Kod_nilma_Subsection_}. \endproof

%%%%%%%%%%%%%%%%%%%%%%%%%%%%%%%%%%%%%%%%%%%%%%%%%%%%%%%%%%%%
\subsection{BG-precursors and principal elliptic
  fibrations}
%%%%%%%%%%%%%%%%%%%%%%%%%%%%%%%%%%%%%%%%%%%%%%%%%%%%%%%%%%%%

%%%%%%%%%%%%%%%%%%%%%%%%%%%%%%%%%%%%%%%%%%%%%%%%%%%%%%%%%%%%
\theorem \label{_diagonal_fibration_Theorem_}
Let $L$ be an ample line bundle on an elliptic
curve, and $\Kod= \Tot^\circ(L)/\langle\lambda\rangle$
a Kodaira surface associated with $(E,L)$. 
Let $p_i:\; E^{n+1} \arrow E$ be the projection to $i$-th
component, ${\cal L}_i:= p_i^*(L)\otimes p_{n+1}^*(L)^{-1}$, and 
$B\subset E^{n+1}$ the divisor $\{(z_1, ...,z_{n+1})\in E^{n+1}\ \ |\ \ \sum z_i =0\}$.
Then the restricted Bogo\-mo\-lov-\-Guan precursor associated with $\Kod$
is naturally isomorphic to 
$R_\lambda(B,{\cal L}_1, {\cal L}_2, ..., {\cal L}_n)$,
and the unrestricted Bogo\-mo\-lov-\-Guan precursor to
$R_\lambda(E^{n+1},{\cal L}_1, {\cal L}_2, ..., {\cal L}_n)$.

\hfill

\proof See 
\ref{_Kodaira-type_manifolds_corollary_}.
\endproof

\hfill

\theorem \label{_Proof_for_class_C_Theorem_}
Let $P:\; X\arrow E^{n+1}$ be the standard fibration
on the unrestricted Bogomolov-Guan precursor, and $Z\subset X$
an irreducible Fujiki class C subvariety.
Then $\dim P(Z)\leq 1$, and
if $\dim P(Z)=1$, then it is a curve
such that all projections $p_i:\;  P(Z)\arrow E$
have the same degree. The converse is also true:
 if $P(Z)$ satisfies these assumptions, 
then $Z$ is Fujiki class C. Moreover, if $Z$ is 
smooth Fujiki class C subvariety in $X$ then it is K\"ahler.

\hfill

\pstep 
Let $\omega_i:= p_i^* \omega$, where $\omega$ is
the K\"ahler class of $E$ which satisfies $\int_E \omega=1$. 
The Chern classes of ${\cal L}_1, {\cal L}_2,{\cal L}_3, ..., {\cal L}_n$ are
proportional to $\omega_i-\omega_{n+1}$.
By \ref{_prop_pullback_of_line_bundle_}, $P^*(\omega_i-\omega_{n+1})=0$
for any $i=1, ..., n$. 
If $Z$ is K\"ahler or Fujiki class C, this implies that
$\omega_i-\omega_{n+1}\restrict{P(Z)}=0$, because
for surjective holomorphic maps of K\"ahler manifolds the pullback
map is injective on rational cohomology
(\cite[p. 177]{_Voisin-Hodge_I_}).

\hfill

{\bf Step 2:} Suppose that $Z$ is Fujiki class C.
By Step 1, $\omega_i-\omega_{j}\restrict{P(Z)}=0$ for all $i, j$.
If $\dim P(Z)>1$, its $i$-th projection to $E$ has
positive-dimensional fiber $S$ for some $i$,
hence $\omega_i \restrict S=0$.
Then $\omega_i-\omega_{j}\restrict{P(Z)}=0$
implies that $\omega_i \restrict S=0$ for all $i$,
which is impossible because $\sum_i \omega_i$ is K\"ahler.

\hfill

{\bf Step 3:} If $P(Z)$ is a curve,
$\omega_i-\omega_{j}\restrict{P(Z)}=0$ implies that
$\int_{P(Z)}\omega_i = \int_{P(Z)}\omega_j$; this number
is equal to the degree of the map $p_i:\;  P(Z)\arrow E$.

\hfill

{\bf Step 4:} It remains to prove that $Z$ is
Fujiki class C if $M:=P(Z)$ is a curve and all projections 
$p_i:\;  P(Z)\arrow E$ have the same degree.
Then $P^{-1}(M)= R_\lambda(M, V_1, ..., V_n)$
where all $V_i$ are line bundles with $c_1(V_i)=0$.
Then $R_\lambda(M, V_1, ..., V_n)$ is Fujiki class C 
by Blanchard theorem  (\ref{_Blanchard_Kahler_Theorem_}).

\hfill

{\bf Step 5:} If $Z$ is a smooth subvariety in $X$ then so is $P(Z)\subset E^{n}$.
Thus, by Blanchard theorem (\ref{_Blanchard_Kahler_Theorem_}) and our previous
argument, we observe that $Z$ is a smooth subvariety in the K\"ahler manifold
$P^{-1}(P(Z)) = R_\lambda(M, V_1, ..., V_n)$. Therefore, $Z$ is K\"ahler.
\endproof

%%%%%%%%%%%%%%%%%%%%%%%%%%%%%%%%%%%%%%%%%%%%%%%%%%%%%%%%%%%%%%%%%%%%%%%%
\subsection{Complex curves in the Bogomolov-Guan precursor}
%%%%%%%%%%%%%%%%%%%%%%%%%%%%%%%%%%%%%%%%%%%%%%%%%%%%%%%%%%%%%%%%%%%%%%%%

The main result of this section is the following theorem.

\hfill

%%%%%%%%%%%%%%%%%%%%%%%%%%%%%%%%%%%%%%%%%%%%%%%%%%%%%%%%%%%%
\theorem\label{_proj_then_quasidiag_Theorem_}
Let $X=\Kod^n/E_1$ be an unrestricted Bogomolov-Guan precursor, and
let $P$ be its projection to the base $E^n$.
Consider an irreducible subvariety $Z\subset X$
 such that $P(Z)$ is a curve. Then 
$Z$ is projective if and only if $P(Z)$ is a
quasi-diagonal.

\hfill

%%%%%%%%%%%%%%%%%%%%%%%%%%%%%%%%%%%%%%%%%%%%%%%%%%%%%%%%%%%%
\remark\label{_proj_dim=1_base_Remark_}
If $\dim P(Z)> 1$, $Z$ cannot be Fujiki class C or Moishezon 
by \ref{_Proof_for_class_C_Theorem_}.
If $\dim P(Z)=0$, the variety $Z$ is projective because all fibers of $P$
are projective. This is why we care only about the case
$\dim P(Z)=1$.

\hfill

\noindent{\bf Proof of \ref{_proj_then_quasidiag_Theorem_}.} {\bf Step 1:}
For any curve $P(Z)$, its preimage in $X$ is identified
with $R_\lambda(M, V_1, ..., V_n)$ by \ref{_diagonal_fibration_Theorem_}.
Let $M$ be a complex curve.
We are going to prove that $R_\lambda(M, V_1, ..., V_n)$
contains an irreducible complex curve not contained in the fiber
of the projection $P:\; R_\lambda(M, V_1, ..., V_n)\to M$
if and only if all line bundles $V_i$ are torsion,
and this is equivalent to $R_\lambda(M, V_1, ..., V_n)$ being projective.

\hfill

{\bf Step 2:} 
Since $R_\lambda(M, V_1, ..., V_n)$ is a fibered product
of $R_\lambda(M, V_i)$, to prove that it is projective
it suffices to show that
$R_\lambda(M, V)=\Tot^\circ(V)/\langle\lambda\rangle$ is 
projective if $V$ is a torsion line bundle,
and contains no irreducible
complex curves, except those in the fibers of $P$, otherwise.

\hfill

{\bf Step 3:} Consider the projection
$P:\; R_\lambda(M, V) \to M$. {\bf A multisection}
of $P$ is an irreducible subvariety $Z\subset R_\lambda(M, V)$
such that $P:\; Z \arrow M$ is finite. 
 For any $d\in \Z^{>0},$ consider a map 
$\Psi:\; R_\lambda(M, V)\times_M ... \times_M R_\lambda(M, V)\arrow R_\lambda(M, V^{\otimes d})$
associated with the product map 
$V\times V \times ... \times V \arrow V^{\otimes d}$.
If we apply this map to a multisection of degree $d$, we obtain a section.
Since $R_\lambda(M, V^{\otimes d})$ is a principal elliptic fibration,
it is trivialized by any section, hence $R_\lambda(M, V^{\otimes d})= M \times E$.
Therefore, it is projective. Restricting $\Psi$ to the diagonal,
we obtain a finite map from $R_\lambda(M, V)$ to a projective manifold,
hence existence of multisections ensures that $R_\lambda(M, V)$ is projective;
the converse is also true: any irreducible curve, not contained in the fiber,
gives a multisection.

\hfill

{\bf Step 4:} 
If $V$ is a torsion, $\Psi:\; R_\lambda(M, V)\arrow R_\lambda(M, V^{\otimes d})= M \times E$
is finite, hence $R_\lambda(M, V)$ is projective. Conversely, if $R_\lambda(M, V)$ is projective,
it admits a multisection, which trivializes the elliptic
fibration 
\[ R_\lambda(M, V^{\otimes d})\arrow M.\]
This completes the proof.
\endproof

\hfill

\remark
\ref{_proj_then_quasidiag_Theorem_}
says that projective submanifolds in 
the unrestricted precursor are preimages
of quasi-diagonals.
Clearly, the projective submanifolds in 
the restricted precursor are preimages
of the restricted quasi-diagonals.

%%%%%%%%%%%%%%%%%%%%%%%%%%%%%%%%%%%%%%%%%%%%%%%%%%%%%%%%%%%%

\section{Barlet spaces}\label{_section_Barlet_spaces_}

%%%%%%%%%%%%%%%%%%%%%%%%%%%%%%%%%%%%%%%%%%%%%%%%%%%%%%%%%%%%

It is possible to study the quasi-diagonals using the tools
of complex  geometry: the Barlet spaces and 
the algebraic coreduction, developed by F. Campana
in \cite{_Campana:coreduction_}
(see also \cite{_Campana_Peternell:Cycle_spaces_}). For the
rest of this paper, we mainly focus on this approach,
discovering interesting complex-geometric properties
of the Bogomolov-Guan precursor. 

%%%%%%%%%%%%%%%%%%%%%%%%%%%%%%%%%%%%%%%%%%%%%%%%%%%%%%%%%%%%
\subsection{The Barlet spaces and the coreduction}
%%%%%%%%%%%%%%%%%%%%%%%%%%%%%%%%%%%%%%%%%%%%%%%%%%%%%%%%%%%%

We start by defining the Barlet space
and listing its basic properties. We follow
\cite{_Magnusson:cycle_,_Barlet_Magnusson_,_Campana_Peternell:Cycle_spaces_}.

\hfill

\definition
{\bf Barlet spaces} are spaces of cycles, that is, 
closed complex analytic subvarieties
of a given dimension in a given complex manifold
with multiplicities (positive integers)
assigned to their irreducible components.
They are similar but distinct from the {\bf Douady spaces}, which are
 spaces of closed complex analytic subspaces (possibly with nilpotents
in the structure sheaf). 

\hfill

\definition
Let $M$ be a metric space. Recall that 
{\bf the Hausdorff metric} on 
the set ${\cal C}$ of closed subsets of $M$
is defined as follows: $d(X,Y)$ is the infimum
of all $\epsilon$ such that $X$ belongs to an
$\epsilon$-neighbourhood of $Y$, and
$Y$ belongs to an $\epsilon$-neighbourhood of $X$.
When $M$ is compact, the corresponding topology
on ${\cal C}$ is independent of  the choice
of metric on $M$ as long as the topology of $M$ 
remains the same. It is called 
{\bf the Hausdorff topology}.

\hfill

\remark
The Barlet space of cycles is complex analytic, with
the topology that is  induced by the Hausdorff
topology on the set of all closed subvarieties
\cite{_Magnusson:cycle_,_Barlet_Magnusson_}.

\hfill

%%%%%%%%%%%%%%%%%%%%%%%%%%%%%%%%%%%%%%%%%%%%%%%%%%%%%%%%%%%%
\theorem\label{_Bishop_compactness_Theorem_}
 (Bishop)\\
Let $(M,I)$ be a compact complex manifold, and
$\omega$ a Hermitian form. Fix $a\in \R$ and let 
$Z_a$ be the  Barlet space ${\mathfrak B}_k(M)$ 
of $k$-cycles $Z\subset M$ such that $\int_Z \omega^k \leq a$.
Then $Z_a$ is compact.

\proof \cite{_Magnusson:cycle_,_Barlet_Magnusson_}.
\endproof

%%%%%%%%%%%%%%%%%%%%%%%%%%%%%%%%%%%%%%%%%%%%%%%%%%%%%%%%%%%%
\subsection{The coreduction map}
%%%%%%%%%%%%%%%%%%%%%%%%%%%%%%%%%%%%%%%%%%%%%%%%%%%%%%%%%%%%

In this subsection, we introduce the coreduction map,
due to F. Campana. We follow
\cite{_Campana:coreduction_,_Campana_Peternell:Cycle_spaces_}.

\hfill
 
\definition
A proper holomorphic map $f:\; X \arrow Y$ of
complex varieties is called {\bf  projective}
if there exists a line bundle $L$ on $X$ which
is ample on all fibers of $f$. It is called
{\bf Moishezon} if there exists a 
bimeromorphism $\mu:\; X \arrow \tilde X$
and a projective morphism $\tilde f:\; \tilde X \arrow Y$
such that $f=  \tilde f \circ \mu$.

\hfill

%%%%%%%%%%%%%%%%%%%%%%%%%%%%%%%%%%%%%%%%%%%%%%%%%%%%%%%%%%%%
\theorem\label{_coreduction_Theorem_}
Let $f:\; X \arrow Y$ be a surjective holomorphic map of 
irreducible compact complex varieties with Moishezon fibers. Assume that all
components of the Barlet space of curves on 
$X$ are compact, and there exists a subvariety $A\subset X$
such that $f(A)=Y$ and $f\restrict A$ is Moishezon.
 Then $X$ is Moishezon.

\proof
\cite{_Campana:coreduction_}.
\endproof

\hfill

This leads to the following theorem

\hfill

\theorem 
Let $X$ be a compact complex variety, such that  all
components of the Barlet space of curves on 
$X$ are compact. Then there exists a meromorphic map
$f:\; X \arrow Y$ with Moishezon fibers, called {\bf the coreduction}
of $X$, such that for a general point $x\in X$ 
any Moishezon subvariety passing through $x$
lies in the fiber of $f$.

\proof
\cite{_Campana:coreduction_}.
\endproof

\hfill

The coreduction construction is surveyed and presented in \cite[Theorem 3.11, Theorem 3.12]{_Campana_Peternell:Cycle_spaces_}.

%%%%%%%%%%%%%%%%%%%%%%%%%%%%%%%%%%%%%%%%%%%%%%%%%%%%%%%%%%%%
\subsection{Gauduchon metric and Gauduchon theorem}
%%%%%%%%%%%%%%%%%%%%%%%%%%%%%%%%%%%%%%%%%%%%%%%%%%%%%%%%%%%%

\definition
Let $(M, I, \omega)$ be a complex Hermitian  $n$-manifold.
The Hermitian metric $\omega$ is called {\bf Gauduchon}
if $dd^c(\omega^{n-1})=0$.

\hfill

The following theorem is one of the most fundamental
and important results in the theory of complex manifolds.
Additionally, 
this is one of the very few known results which work for all compact complex
manifolds, and not for some special classes.

\hfill

\theorem (Paul Gauduchon, 1977)\\
Let $(M, I, \omega)$ be a compact complex Hermitian $n$-manifold.
Then  there exists a unique (up to a constant factor)
positive smooth function $e^f$ such that $e^f\omega$ is 
Gauduchon.

\proof \cite{_Gauduchon_1984_}, see also \cite[Chapter 23]{_OV:Principles_}. \endproof

\hfill

%%%%%%%%%%%%%%%%%%%%%%%%%%%%%%%%%%%%%%%%%%%%%%%%%%%%%%%%%%%%
\example\label{_Kodaira_Gauduchon_Example_}
Let $M= G/\Gamma$ be the Kodaira surface realized as
a nilmanifold (Subsection \ref{_Kod_nilma_Subsection_}), with $x,y, z, t$ the generators of its Lie algebra,
and the only non-trivial commutator $[x,y]=z$,
and $a,b,c,d$ the dual 1-forms. The de Rham differential
is a dual map to the commutator. Consider the complex
of left-invariant differential forms on $G$, identified
with $\Lambda^* \g$. Using the Cartan formula, we
can express de Rham differential on $\Lambda^* \g$
through the commutator in $\g$. On $\Lambda^1\g=\g^*$,
de Rham differential is the dual map to $[\cdot,\cdot]:\; \Lambda^2 \g \to \g$.
Therefore, the only non-trivial differential on
$\Lambda^1\g=\g^*$ is given by $d c=a \wedge b$. This 
gives a formula for the left-invariant
Gauduchon metric: $\omega=a\wedge b + c\wedge d$.

%%%%%%%%%%%%%%%%%%%%%%%%%%%%%%%%%%%%%%%%%%%%%%%%%%%%%%%%%%%%
\subsection{SKT manifolds and Barlet spaces}
%%%%%%%%%%%%%%%%%%%%%%%%%%%%%%%%%%%%%%%%%%%%%%%%%%%%%%%%%%%%

\definition
A Hermitian form $\omega$ on a complex manifold is {\bf SKT}
(strong K\"ahler torsion) or {\bf pluriclosed}
if $dd^c \omega =0$.

\hfill

The following theorem is a 1-dimensional version of 
Bishop's theorem (\ref{_Bishop_compactness_Theorem_});
however, it works for the space of
pseudo-holomorphic curves on almost complex manifolds.

\hfill

%%%%%%%%%%%%%%%%%%%%%%%%%%%%%%%%%%%%%%%%%%%%%%%%%%%%%%%%%%%%
\theorem (Gromov)\\
Let $M$ be a compact Hermitian almost complex manifold,
${\goth X}$ the Barlet space of all complex curves on $M$, and
${\goth X}\stackrel \Vol \arrow \R^{>0}$ the volume
function. Then $\Vol$ is {\bf  proper} (preimage of
a compact set is compact).

\hfill

The following corollary was originally obtained in
\cite{_Verbitsky:rat-twi_}, see also \cite{_Verbitsky:S^6_}.

\hfill

%%%%%%%%%%%%%%%%%%%%%%%%%%%%%%%%%%%%%%%%%%%%%%%%%%%%%%%%%%%%
\corollary\label{_SKT_then_Barlet_compact_Corollary_}
Let $M$ be a complex manifold, equipped with a pluriclosed
Hermitian form $\omega$, and $X$ a component of the moduli of
complex curves. Then the function $\Vol:\; X \arrow \R^{>0}$
is constant, and $X$ is compact.

\hfill

\proof
Since $\Vol\geq 0$, the set
$\Vol^{-1}(]-\infty, C])$ is compact for all $C\in \R$,
the function $\Vol$ has a minimum somewhere in $X$. However, a
pluriharmonic function which has a minimum is necessarily constant
by E. Hopf's strong maximum principle (\cite{_Gilbarg_Trudinger_}). Therefore,
$\Vol$ is constant: $\Vol=A$. Now,
compactness of $X= \Vol^{-1}(A)$ follows from Gromov's theorem. \endproof

\hfill

Later in this section,
we are going to prove that the Bogomolov-Guan precursor is
SKT, and apply \ref{_SKT_then_Barlet_compact_Corollary_} to its Barlet space.

\hfill

Let $\omega= a\wedge b + c\wedge d$ be the Gauduchon metric on $\Kod$,
written in terms of the frame associated with its Lie algebra
as above, and  $a_i, b_i, c_i, d_i$, $i=1, ..., n+1$ 
the corresponding 1-forms on $\Kod^{n+1}$, 
obtained in the same way as in the nilmanifold
description of the Kodaira surface (Subsection
\ref{_Kod_nilma_Subsection_}), 
\[ d(c_i)=a_i\wedge b_i,\ \  d(a_i)=d(b_i)=d(d_i)=0.
\] Consider the
Hermitian form $\omega:=\sum_i a_i\wedge b_i + c_i\wedge d_i$
on $\Kod^{n+1}$. Since $\Kod$ is Gauduchon, $dd^c\omega=0$.

\hfill

\remark
 Let $X_1:= \Kod^{n+1}/E_1$, where $E_1$ acts diagonally,
and $\Kod^{n+1} \stackrel \pi \arrow X_1$ the corresponding
holomorphic projection. Denote by 
\[ \pi_*:\; \Lambda^{p,q}(\Kod^{n+1})\arrow \Lambda^{p-1,q-1}(X_1)\]
the pushforward, that is, the fiberwise integration of differential
forms. 
Clearly, $\pi_*(\omega^2)$ is a Hermitian form on $X_1$.
Since $\pi_*$ commutes with $d, d^c$,
to show that $dd^c\pi_*(\omega^2)=0$
it would suffice to show that the (2,2)-form
\[
dd^c\pi_*(\omega^2)= \pi_*(dd^c \omega\wedge \omega)
+ \pi_*(d\omega\wedge d^c \omega) = \pi_*(d\omega\wedge d^c \omega)
\]
vanishes.

\hfill

The following proposition immediately implies that $X_1$,
and hence the Bogomolov-Guan precursor $X\subset X_1$, is an SKT manifold.

\hfill

\proposition
Let $\Kod^{n+1} \stackrel \pi \arrow X_1$ be the 
holomorphic projection and $\omega$  be the SKT Hermitian
form defined above. Then $\pi_*(d\omega\wedge d^c
\omega)=0$.

\hfill

\proof
For any invariant differential form, expressed as a 
Grassmann polynomial of $a_i, b_i, c_i, d_i$, the fiberwise
integration is given by the operator $\sum_{i=1}^{n+1} \iota_{z_i} \iota_{t_i}$,
where $\iota$ denotes the convolution with a vector field, and
$z_i, t_i$ the generators of the vector fields tangent
to the components of $\Kod^{n+1}$ associated with the
frame $x,y, z,t\in T\Kod$ dual to $a,b,c,d \in \Lambda^1\Kod.$
Clearly, $d\omega =\sum_{i=1}^{n+1} a_i \wedge b_i \wedge d_i$,
and $d^c\omega =-\sum_{i=1}^{n+1} a_i \wedge b_i \wedge c_i$,
hence $d\omega\wedge d^c \omega= \sum_{i<j} a_i \wedge b_i\wedge d_i\wedge a_j \wedge b_j\wedge c_j$.
After contracting with the bivector 
$\sum_i z_i\wedge t_i$ dual to $\sum_i c_i \wedge d_i$,
this form vanishes. \endproof

\hfill

We have established the following theorem.

\hfill

%%%%%%%%%%%%%%%%%%%%%%%%%%%%%%%%%%%%%%%%%%%%%%%%%%%%%%%%%%%%
\theorem \label{_precursor_SKT_Theorem_}
The Bogomolov-Guan precursor manifold (restricted and unrestricted) 
admits an SKT metric.
\endproof

%%%%%%%%%%%%%%%%%%%%%%%%%%%%%%%%%%%%%%%%%%%%%%%%%%%%%%%%%%%%
\subsection{Countability of the number of quasi-diagonals obtained from
  the coreduction}
\label{_Countability_coreduction_Subsection_}
%%%%%%%%%%%%%%%%%%%%%%%%%%%%%%%%%%%%%%%%%%%%%%%%%%%%%%%%%%%%

In this subsection, we apply \ref{_proj_then_quasidiag_Theorem_},
\ref{_SKT_then_Barlet_compact_Corollary_}
and \ref{_precursor_SKT_Theorem_} to show that
the number of $(E,L)$-quasi-diagonals is countable
for any given $E$ and $L$; an algebraic proof
of this result is given in 
\ref{_corollary:_line_bundles_for_restricted_qd_Corollary}.

\hfill

%%%%%%%%%%%%%%%%%%%%%%%%%%%%%%%%%%%%%%%%%%%%%%%%%%%%%%%%%%%%
\theorem\label{_quasidiag_countable_via_SKT_Theorem_}
There are at most countably many quasi-diagonals for any given $(E,L)$.

\hfill

\pstep
Since quasi-diagonals in $E^n$ are obtained from
quasi-diagonals in $E^2$ (\ref{_quasi-diagonals_in_En_Theorem_}),
it suffices to show that there are at most countably 
many quasi-diagonals in $E^2$.
If the number of quasi-diagonals is uncountable,
there are continuous families of quasi-diagonals, which cover $E^2$.
We make this assumption and show that it leads to contradiction.

\hfill

{\bf Step 2:}
Denote by $B$ a positive-dimensional 
irreducible component of the Barlet space of
quasi-diagonals in $E^2$,
and let $C_t$, $t\in B$ denote the quasi-diagonals.
Then ${\cal L}_1$ and ${\cal L}_2$ are torsion on each $C_t$
and satisfy ${\cal L}_1^{\otimes d}=\calo_{C_t}$ for some fixed $d$ and all $t$.
Let $P:\; X \arrow E^2$ be the unrestricted 
Bogomolov-Guan precursor associated with $L$; by 
\ref{_Kodaira-type_manifolds_Lemma_}, we have
$X= R_\lambda(E^2, {\cal L}_1\otimes {\cal L}_2^{-1})$.
Consider the degree $d^2$ map
$P_d:\; X\arrow R_\lambda(E^2, {\cal L}_1^{\otimes d},
{\cal L}_2^{\otimes d})$.
By \ref{_diagonal_fibration_Theorem_},  $P_d^{-1}(C_t) \stackrel {P_d} \arrow C_t$
is a trivial $E_1^2$-fibration. Therefore, it admits a family
of sections parametrized by $E_1^2$. These sections
produce degree $d^2$ multisections of 
$P:\; R_\lambda(E^2, {\cal L}_1, {\cal L}_2)\arrow E^2$,
passing through every point of $X=R_\lambda(E^2, {\cal L}_1, {\cal L}_2)$.

\hfill

{\bf  Step 3:}
Using the curves constructed in Step 2, we obtain
 a collection of curves connecting any two points
in $X$. Using the coreduction theorem of Campana (\ref{_coreduction_Theorem_}),
we obtain that $X$ is Moishezon, contradicting
\ref{_Proof_for_class_C_Theorem_}. 
\endproof
 
\hfill

\corollary
Let $X$ be a restricted Bogomolov-Guan precursor associated
with a general $(E,L)$. Then all projective, 
Moishezon subvarieties of $X$ belong to the
fiber of the Lagrangian projection.

\hfill

\proof 
By \ref{_proj_then_quasidiag_Theorem_}, all 
 projective, Moishezon subvarieties in a restricted Bogomolov-Guan precursor
project to restricted $(E,L)$-quasi-diagonals in $E^n$.
By \ref{_corollary:_line_bundles_for_restricted_qd_Corollary}, 
for a general $L$ 
there are no restricted $(E,L)$-quasi-diagonals.
\endproof

\hfill

%%%%%%%%%%%%%%%%%%%%%%%%%%%%%%%%%%%%%%%%%%%%%%%%%%%%%%%%%%%%

\appendix
\section{Preimages of lines in 4-dimensional BG} \label{section: no quasi-diagonals over lines}
In this section, we consider the Lagrangian fibration of a $4$-dimensional Bogomolov--Guan manifold $P\colon BG \to \p^2$.
One can ask whether a line on the Lagrangian base $\p^2$ is the image of a restricted quasi-diagonal on $E^3$ for some line bundle 
$L$ on $E$. \ref{_restricted_quasi-diagonal_for_3_Example_} provides an example of this situation.

\hfill

\begin{lemma}\label{_lemma_restricted_qd_in_E3_maps_to_line_}
Let $S\subset A = \{(x,y,z)\mid x+y+z = e \} \subset E^3$ be 
the curve defined in \ref{_restricted_quasi-diagonal_for_3_Example_}. 
Then $\alpha(S)$ is a line on $\p^2$.
\end{lemma}

\hfill

\proof
 The group $\langle\zeta\rangle$ acts diagonally on $E^3$. 
 This action preserves the divisor $A\subset E^3$ and 
 commutes with the action of symmetric group $\SG_3$.
 Thus, $\langle \zeta\rangle$ acts on $A/\SG_3\cong \p^2$,
 one can check that this action is faithful.

 Consider the image of the curve $S\subset A$ under $\alpha$.
 One can observe that each point in $\alpha(S)$ is fixed 
 under the action of $\langle \zeta\rangle$  on $\p^2$.
 Therefore, $\alpha(S)$ is a pointwise fixed curve 
 under a non-trivial action of $\langle \zeta\rangle$.
 Thus, it is a line. 
\endproof

\hfill

The situation described in \ref{_restricted_quasi-diagonal_for_3_Example_} is very specific.
We are going to check whether a preimage of a line in $\p^2$ could be a smooth curve in $E^3$
and whether one can find a line bundle on $E$ such that this smooth curve is a quasi-diagonal.

We fix the notation: let $e\in E$ be the neutral point and denote by 
$p_i\colon E^3\to E$ the projection to the $i$-th component of the product. 
Denote by $A\subset E^3$ the divisor $A = \{(x_1,x_2,x_3)\in E^3\mid x_1+x_2+x_3 = e\}$. 
Denote by $\alpha\colon A\to \p^2$ the symmetrization map. The ramification divisor of $\alpha$
can be described in the following way.

\hfill

\begin{lemma}\label{_lemma_duality_}
  Consider the embedding by the linear system 
  \[
  \varphi = \varphi_{|\OP_E(3e)|}\colon E\to \p(H^0(E,\OP_E(3e)))^{\vee}\cong (\p^2)^{\vee}.
  \]
  The curve $\varphi(E)\subset  (\p^2)^{\vee}$ is projectively dual to the ramification  divisor of $\alpha$ in $\p^2$. 
  Moreover, for each point $(x,y,z)\in A$ its image $\alpha((x,y,z))\in \p^2$  corresponds to the line passing through points $\varphi(x)$,
  $\varphi(y)$ and $\varphi(z)$ in $ (\p^2)^{\vee}$.
\end{lemma}

\hfill

\proof
 \cite[Lemma 7.5]{_BKKY_}.
\endproof

\hfill

Fix a line $\l\subset \p^2$ and consider its preimage $S_{\l} = \alpha^{-1}({\l})$. 
Then $S_{\l}$ is an $\SG_3$-invariant curve in $A$. Further, we deduce that $S_{\l}$ is 
a smooth curve if $\l$ is general, see \ref{lemma: RR}. However, if $\l$ is tangent
to the ramification divisor of $\alpha$ then its preimage is reducible.

\hfill

\begin{lemma} \label{_lemma_class_of_preimage_of_line_in_A_}
Let ${\l}$ be a line tangent to the ramification divisor of the map~$\alpha$. Then there exists a point $x\in E$ such that
\[S_{\l}
= \left(p_1^{-1}(x)\cup p_2^{-1}(x)\cup
p_3^{-1}(x)\right)\cap A.
\]
\end{lemma}
\proof
 This follows from the description of the ramification divisor of $\alpha$ given in \ref{_lemma_duality_}.
\endproof

 %The line bundle
%$\alpha^*(\OP_{\p^2}(1))$ is ample; thus, by Bertini
%theorem $S_{\l}$ is a smooth curve for a general
%line~${\l}$. 

\hfill

\lemma \label{lemma: RR} Each element in 
the linear system $|\alpha^*(\OP_{\p^2}(1))|$ is the preimage of a line $\l\subset \p^2$ 
and a general element  of $|\alpha^*(\OP_{\p^2}(1))|$ is smooth.

\hfill

\proof 
Denote the line bundle $\alpha^*(\OP_{\p^2}(1))$ by $L$.
Denote by $F_i \in NS(A)$ the class of fiber $p_i^{-1}(x)\cap A$. 
On the surface $A$ one has $F_i^2 = 0$
and $F_i\cdot F_j = 1$. 

By  \ref{_lemma_class_of_preimage_of_line_in_A_}  we observe that $c_1(L) =
F_1+F_2+F_3$. Then by Riemann-Roch theorem we compute 
(note that $A$ is a torus, hence its Chern classes
vanish and don't contribute to the Riemann-Roch-Hirzebruch formula):
\[
h^0(A,L)-h^1(A,L)+h^2(A,L) = \chi(L) = \frac{1}{2}(F_1+F_2+F_3)^2 = 3.
\]
Moreover, since for distinct $x\in E$  the preimages of
$\l_x$ don't have common points,
the line bundle $L$ has empty base locus.
Thus, it is ample and $H^1(A,L) = H^2(A,L) = 0$ by Kodaira
vanishing on $A$. This implies that $H^0(A,L) = H^0(\p^2,\OP_{\p^2}(1))$.
Thus, the result follows by Bertini theorem.
\endproof

\hfill

The next assertion is the main result of this section. 
It is followed by several technical lemmas which are used in its proof.

\hfill

\proposition \label{_proposition_S_l_is_not_quasi-diagonal_Lemma_} 
If ${\l}$ is a line on $\p^2$ such that $S_{\l}$ is smooth, then $S_{\l}$ is not a quasi-diagonal for any line bundle $L$ on the curve $E$. 

\hfill

\proof
By \ref{lemma: RR} for a general line $\l$ the curve $S_{\l}$ is smooth.
Assume that $y$ is a point on $E$ 
such that the curve $S_{\l}$ is a quasi-diagonal with
respect to the line bundle $\calo([y])$.
Then, for certain $n\in \Z^{>1}$, we have
$\left(p_1^*L\otimes p_2^*L^{-1}\right)|_{S_{\l}} \cong \OP_{S_{\l}}$, 
here $L=\calo(n[y])$. Since replacing $n$ by its multiple preserves
the properties of the quasi-diagonals, we further assume that $n\geqslant 3$.

Denote by $\mathcal{F}_1, \mathcal{F}_2, \mathcal{F}_3$ the following sheaves on $A$:
\begin{align*}
&\mathcal{F}_1 = p_1^*\OP(n[y] - [x])\otimes p_2^*\OP(-n[y] - [x])\otimes p_3^*\OP(- [x]);\\
&\mathcal{F}_2 = p_1^*\OP(n[y])\otimes p_2^*\OP(-n[y]);\\
&\mathcal{F}_3 = \left(p_1^*L\otimes p_2^*(L^{-1})\right)|_{S_{\l}}.
\end{align*}
These sheaves fit the following exact sequence on the surface $A$:
\begin{equation}\label{_exact_sequence_for_qd_Equation_}
0\to \mathcal{F}_1 \to \mathcal{F}_2 \to \mathcal{F}_3\to 0.
\end{equation}
If $S_{\l}$ is a quasi-diagonal then 
$\mathcal{F}_3 = \OP_{S_{\l}}$; thus, $H^0(A,\mathcal{F}_3) = \mathbb{C}$ since $S_{\l}$ is smooth. 
On the other hand, if $S_{\l}$ is not a quasi-diagonal 
then $H^0(A,\mathcal{F}_3) = 0$.
Thus, by \ref{_lemma_cohomology_of_F1} and \ref{_lemma_cohomology_of_F2} 
below in order to check whether $S_{\l}$ is a quasi-diagonal, it suffices 
to check that the map $H^1(A, \mathcal{F}_1)\to H^1(A, \mathcal{F}_2)$ is injective.
Since this map is indeed injective by 
\ref{lemma: proof of injectivity}, we conclude the result.
\endproof

\hfill

\lemma \label{_lemma_cohomology_of_F1} Let $x$ and $y$ be two points on the elliptic curve $E$ and let $n\geqslant 3$ be a positive integer.
Then the cohomology groups of the line bundle $\mathcal{F}_1 = p_1^*\OP(n[y] - [x])\otimes p_2^*\OP(-n[y] - [x])\otimes p_3^*\OP(- [x])$ on $A$
are as follows:
\[
H^i(A, \mathcal{F}_1 ) = \left\{\begin{aligned}
&0, && \text{ if $i = 0$ or $2;$}\\
&\C^{n^2-3}, && \text{ if $i = 1$.}
\end{aligned}\right.
\]

\hfill

\proof
 Fix a point $a\in E$, consider the elliptic curve $C_1 = p_1^{-1}(a)\cap A\subset A$ and restrict 
 $\mathcal{F}_1$ to $C_1$. One can check that $\deg(\F_1|_{C_1}) = -n-2$. Since the class of $C_1$ 
 is nef on $A$ we deduce that $H^0(A,\mathcal{F}_1) = 0$.

 Consider the curve $C_2 = p_2^{-1}(a)\cap A\subset A$ and restrict the dual line bundle
 $\mathcal{F}_1^{-1}$ to $C_2$. One can check that $\deg(\F_1^{-1}|_{C_2}) = 2-n<0$.
 Since the class of the curve $C_2$ is nef on $A$ by Serre duality, we conclude:
 \[
 H^2(A,\F_1) = H^0(A,\F_1^{-1})^{\vee} = 0.
 \]

 Now denote by $F_i\subset E^3$ the preimage of a point $x\in E$ under 
 the projection to the $i$-th component of the product.
 The Neron--Severi group of $A\cong E^2$ contains classes 
 $f_1, f_2, f_3$ such that $f_i$ is the class of the divisor $F_i\cap A$.
 One can check that $f_i^2 = 0$ for all $i = 1,2,3$ and $f_i\cdot f_j = 1$ 
 for all $1\leqslant i<j\leqslant 3$. Then by the Riemann--Roch formula, one 
 has the following equality:
 \[
 \chi(\F_1) = \frac{\left((n-1)f_1  - (n+1)f_2 - f_3\right)^2}{2} = 3-n^2.
 \]
 Since we already have $H^0(A,\mathcal{F}_1) = H^2(A,\mathcal{F}_1) = 0$, the proof is complete.
\endproof

\hfill

\lemma \label{_lemma_cohomology_of_F2} Let $y$ be a point on the elliptic curve $E$ and let $n\geqslant 1$ be a positive integer.
Then the cohomology groups of the line bundle $\mathcal{F}_2 =  p_1^*\OP(n[y])\otimes p_2^*\OP(-n[y])$ on $A$
are as follows:
\[
H^i(A,\mathcal{F}_2 ) = \left\{\begin{aligned}
&0, && \text{ if $i = 0$ or $2;$}\\
&\C^{n^2}, && \text{ if $i = 1$.}
\end{aligned}\right.
\]

\hfill 

\proof
This follows from the K\"unneth formula.
\endproof

\hfill

\begin{lemma}\label{lemma: proof of injectivity}
The map $H^1(A, \mathcal{F}_1)\to H^1(A, \mathcal{F}_2)$ induced by the exact sequence \eqref{_exact_sequence_for_qd_Equation_} is injective.
\end{lemma}

\hfill

\proof 
Consider the following line bundle $\mathcal{G}$ on $A$:
\[
\G = p_1^*\OP(n[y]-[x])\otimes p_2^*\OP(-n[y]-[x]).
\]
Denote by $\varphi\colon \F_1\to \G$ and $\psi\colon \G\to \F_2$ the natural embeddings of sheaves.
The composition $\psi\circ \varphi\colon \F_1\to \F_2$ is the embedding from the exact sequence 
\eqref{_exact_sequence_for_qd_Equation_}.

Consider the following short exact sequence:
\[
0\to \F_1\xrightarrow{\varphi} \G\to \G|_{p_3^{-1}(x)} \to 0.
\]
The degree of the line bundle $\G|_{p_3^{-1}(x)}$ on the curve $p_3^{-1}(x)$ equals $-2$. 
Thus, we observe that $H^0(A,\G|_{p_3^{-1}(x)} ) $ vanishes  and the map $\varphi$ induces the embedding $H^1(A,\F_1)\to H^1(A,\G)$.

Since both line bundles $\G = \OP(n[y]-[x])\boxtimes \OP(-n[y]-[x])$ and 
$\F_2 = \OP(n[y])\boxtimes \OP(-n[y])$ on $A$ are box products
then their cohomology groups can be computed with the K\"unneth formula. 
Note that $\psi$ is induced by embeddings of the components of the box products.
Thus, $\psi$ induces the embedding $H^1(A,\G)\to H^1(A,\F_2)$.
This completes the proof.
\endproof

\hfill

{\bf Acknowledgments:} 
We are grateful to Andrey Soldatenkov and Olivier
Martin for interesting discussions and comments.

\hfill

\scriptsize

\noindent
{\sc Leila Abubakarova\\
{\sc National Research University Higher School of Economics, Moscow, Russia; \\
École normale supérieure, Paris, France.\\
\tt leila.abubakarova.a@gmail.com}
}

\hfill

\noindent
{\sc Alexandra Kuznetsova\\ 
{\sc Higher School of Modern Mathematics in Moscow Institute of Physics and Technology, Dolgoprudny, Russia; \\
     Laboratory of Algebraic Geometry in Higher School of Economics, Moscow, Russia;  \\
     Steklov Mathematical Institute of Russian Academy of Sciences, Moscow, Russia.\\
\tt sasha.kuznetsova.57@gmail.com}     
}

\hfill

\noindent {\sc Misha Verbitsky\\
{\sc Instituto Nacional de Matem\'atica Pura e
              Aplicada (IMPA) \\ Estrada Dona Castorina, 110\\
Jardim Bot\^anico, CEP 22460-320\\
Rio de Janeiro, RJ - Brasil\\
\tt  verbit@impa.br }
}

\end{document}